\documentclass[12pt,a4paper]{article}
\usepackage[margin=2.8cm,footskip=1cm]{geometry}
\usepackage[utf8]{inputenc}
\usepackage[T1]{fontenc}
\usepackage[english]{babel}
\usepackage{amsmath}
\usepackage{amsfonts}
\usepackage{amssymb}
\usepackage{amsthm}
\usepackage{enumitem}
\usepackage{ stmaryrd }
\usepackage{ucs}
\usepackage{graphicx}
\usepackage{xcolor}
\usepackage{dsfont}
\usepackage{tikz}
\usepackage{wasysym}
\usepackage{physics}
\usepackage{mathtools}
\usepackage{xifthen}
\usepackage[hidelinks]{hyperref}
\usepackage{authblk}
\usepackage{url}
\usepackage{algorithm2e}

\renewcommand{\norm}[1]{\|#1\|}

\newcommand{\normsup}[1]{\left\|#1\right\|_{\scriptscriptstyle\infty}}

\newcommand{\given}{\,|\,}

\newcommand{\dTV}{\mathrm{d}_{\mathrm{TV}}}

\newcommand{\asEq}{\stackrel{\mathrm{a.s.}}{=}}
\newcommand{\lawEq}{\stackrel{\mathrm{law}}{=}}

\newcommand{\refBloc}{\mathrm{ref}}

\newcommand{\bnd}{\mathrm{bnd}}
\newcommand{\bulk}{\mathrm{bulk}}
\newcommand{\ext}{\mathrm{ext}}

\newcommand{\MarkovSet}{\mathrm{Ma}}
\newcommand{\MarkovEl}{\mathrm{a}}

\newcommand{\Cmix}{C_{\mathrm{mix}}}
\newcommand{\cmix}{c_{\mathrm{mix}}}

\newcommand{\pbulk}{p_{\mathrm{bulk}}}

\newcommand{\good}{\mathrm{G}}

\newcommand{\Aand}{\textnormal{\textbf{ and }}}

\newcommand{\HardCore}{\mathrm{HC}}

\newcommand{\N}{\mathbb{N}}
\newcommand{\Z}{\mathbb{Z}}
\newcommand{\R}{\mathbb{R}}

\newcommand{\bbE}{\mathbb{E}}
\newcommand{\bbL}{\mathbb{L}}
\newcommand{\bbN}{\mathbb{N}}

\newcommand{\rmB}{\mathrm{B}}

\newcommand{\rmP}{\mathrm{P}}

\newcommand{\rmb}{\mathrm{b}}
\newcommand{\rmd}{\mathrm{d}}
\newcommand{\rme}{\mathrm{e}}
\newcommand{\rmi}{\mathrm{i}}
\newcommand{\rmp}{\mathrm{p}}
\newcommand{\rms}{\mathrm{s}}

\newcommand{\rmE}{\mathrm{E}}

\newcommand{\rmW}{\mathrm{W}}

\newcommand{\calA}{\mathcal{A}}
\newcommand{\calB}{\mathcal{B}}
\newcommand{\calC}{\mathcal{C}}
\newcommand{\calD}{\mathcal{D}}

\newcommand{\calF}{\mathcal{F}}
\newcommand{\calG}{\mathcal{G}}

\newcommand{\calJ}{\mathcal{J}}

\newcommand{\calM}{\mathcal{M}}

\newcommand{\calP}{\mathcal{P}}
\newcommand{\calR}{\mathcal{R}}
\newcommand{\calS}{\mathcal{S}}
\newcommand{\calT}{\mathcal{T}}
\newcommand{\calV}{\mathcal{V}}
\newcommand{\calW}{\mathcal{W}}
\newcommand{\calX}{\mathcal{X}}

\theoremstyle{plain}
\newtheorem{theorem}{Theorem}[section]
\newtheorem{lemma}[theorem]{Lemma}

\newtheorem{remark}{Remark}[section]

\newtheorem{claim}{Claim}

\theoremstyle{definition}
\newtheorem{definition}{Definition}
\newtheorem*{definition*}{Definition}
\newtheorem{example}{Example}[section]

\title{A new perspective on the equivalence between Weak and Strong Spatial Mixing in two dimensions}

\author[1]{S\'{e}bastien Ott\thanks{ott.sebast@gmail.com}}
\affil[1]{Institute of Mathematics, EPFL}

\date{\today}

\begin{document}

	\maketitle
	
	\begin{abstract}
		Weak mixing in lattice models is informally the property that ``information does not propagate \emph{inside} a system''. Strong mixing is the property that ``information does not propagate \emph{inside and on the boundary of} a system''. In dimension two, the boundary of reasonable systems is one dimensional, so information should not be able to propagate there. This led to the conjecture that in 2D, weak mixing implies strong mixing. The question was investigated in several previous works, and proof of this conjecture is available in the case of finite range Gibbsian specifications, and in the case of nearest-neighbour FK percolation (under some restrictions). The present work gives a new proof of these results, extends the family of models for which the implication holds, and, most interestingly, provides a ``percolative picture'' of the information propagation.
	\end{abstract}
	
	\setcounter{tocdepth}{1}
	\tableofcontents
	
	\section{Introduction and results}
	
	\subsection{Mixing in lattice models}
	
	Mixing in lattice models has a long history that will not be entirely reviewed here. I will focus on the parts of the picture relevant for the present paper. A classical quantitative notion of spatial mixing is the notion of \emph{weak mixing}. Informally: exponential weak mixing says that for finite volume measures, \(\mu_{\Lambda}\), one has an estimate of the following form:
	\begin{equation*}
		\dTV\big(\mu_{\Lambda}(X_{\Delta_1}\in \cdot \given X_{\Delta_2} = a), \mu_{\Lambda}(X_{\Delta_1}\in \cdot \given X_{\Delta_2} = b)\big) \leq C\sum_{x\in \Delta_1}\sum_{y\in \Delta_2\cup \Lambda^c} e^{-c|x-y|}.
	\end{equation*}where \(\dTV\) is the total variation distance, \(\Delta_1,\Delta_2\subset \Lambda\). In words: the dependency of the field restricted to \(\Delta_1\) on the field restricted to \(\Delta_2\) is decaying exponentially in the distance between \(\Delta_1\) and \emph{the closest between \(\Delta_2\) and \(\Lambda^c\)}. In other words, it allows for arbitrarily far information transfer through the boundary of the system.
	
	Forbidding this information transfer is obtained using the stronger notion of \emph{exponential strong mixing} which is an estimate of the form
	\begin{equation*}
		\dTV\big(\mu_{\Lambda}(X_{\Delta_1}\in \cdot \given X_{\Delta_2} = a), \mu_{\Lambda}(X_{\Delta_1}\in \cdot \given X_{\Delta_2} = b)\big) \leq C\sum_{x\in \Delta_1}\sum_{y\in \Delta_2} e^{-c|x-y|}.
	\end{equation*}
	This stronger notion, also called \emph{complete analyticity} (CA), is extremely powerful: indeed, it has a wealth of equivalent formulations in terms of analytic properties of the model, existence of expansions for correlation functions, and fast relaxation of the Glauber dynamic associated to the model, see \cite{Dobrushin+Shlosman-1985,Dobrushin+Shlosman-1987,Dobrushin+Warstat-1990,Stroock+Zegarlinski-1992,Yoshida-1997}. This notion is relatively easy to establish in many models under a ``very high temperature'' assumption, using cluster expansion. Non-perturbative results are much harder to obtain, and usually require a lot of additional structure. Two famous cases are: 1) the monomer-dimer model on \(\Z^d\),~\cite{vandenBerg-1999}, and 2) the Ising model on \(\Z^d\) above the critical temperature,~\cite{Ding+Song+Sun-2023}.
	
	A very special case is the case of dimension two. Indeed, in this case, the boundary of a (reasonable) system is one-dimensional, and one dimensional systems famously have systematic exponential decay of every kind of correlations one can think of. So, it is natural to expect that systems satisfying weak mixing (exponential decay of information in the bulk) in dimension two (with a one-dimensional boundary left to perform information transfer) also satisfy the strong mixing condition.
	Note that one has first to make sense of ``reasonable system'' to be able to expect anything: if a volume boundary is too fractal at the scale of the correlation length of the system (or even worse: at the scale of the interaction!), there is no hope that ``weak implies strong'' will hold generically. The relevant notion is called \emph{restricted complete analyticity} (RCA): in terms of mixing, it asks that strong mixing holds for disjoints unions of sufficiently large squares. See~\cite{Martinelli-1999, Martinelli+Olivieri+Schonmann-1994} for more details on this notion, and the discussion in the introduction of~\cite{vanEnter+Fernandez+Schonmann+Shlosman+1997} for more on the failure of the full CA in models as nice as the Potts or Ising models on \(\Z^2\) at values of the parameters for which weak mixing holds.
	
	The equivalence between weak mixing and RCA was established in a large family of systems (finite range Gibbsian specifications) in~\cite{Martinelli+Olivieri+Schonmann-1994}, based on the dynamical criterion of~\cite{Stroock+Zegarlinski-1992}. Strong mixing for more particular models in dimension two under an ``off-criticality assumption'' was established in~\cite{Schonmann+Shlosman-1995} (Ising models, 28 years before~\cite{Ding+Song+Sun-2023}),~\cite{vanEnter+Fernandez+Schonmann+Shlosman+1997} (RCA for Potts model on \(\Z^2\) with large enough \(q\) above \(T_c\)), and~\cite{Alexander-2004} (Potts models, and FK-percolation on \(\Z^2\), under a simply connected assumption on the volume). In the case of FK percolation, one additionally has to be careful with what is meant by ``strong mixing'', and by ``one-dimensional boundary conditions'', as FK-percolation has non-local weights (see the discussion in the Introduction of~\cite{Alexander-2004}).
	
	The present paper has two main objectives: 1) in the spirit of~\cite{Alexander-2004}, one goal is to find models beyond finite range Gibbsian specifications for which strong mixing holds, 2) obtain a ``true'' percolative picture of how information is transmitted through the system. The main result of this work is therefore a bound relating dependency between different parts of a given volume and a percolation event in a \emph{very subcritical Bernoulli percolation model with inhomogeneities at the boundary of the system} (which are one-dimensional and therefore do not break exponential decay of connectivities).
	
	Sections~\ref{sec:applications:FRGibbs},~\ref{sec:applications:FK}, and~\ref{sec:applications:HC} respectively give applications of the main result to finite range Gibbsian specifications, finite range FK percolation on \(\Z^2\), and hard core models on \(\Z^2\). Application to the Random Cluster representation of the Ashkin-Teller model is discussed in~\cite{Dober+Glazman+Ott-2025}.
	
	It is also worth stressing that once strong mixing is established, one can use the beautiful arguments of~\cite[section 5]{Alexander-1998} as in~\cite[section 3]{Alexander-2004} to improve this to \emph{ratio strong mixing} (see the first section of~\cite{Alexander-2004}). Indeed, the hypotheses of the present work imply that the model has exponentially bounded blocking/controlling regions in the sense of~\cite{Alexander-1998,Alexander-2004} which is the key to go from strong mixing to ratio strong mixing. This particular direction being well explored in~\cite{Alexander-1998,Alexander-2004}, the focus of the present work will be limited to establishing strong mixing estimates.

	\subsection{Notation, constants, surrounding sets}
	
	\vspace*{3pt}
	
	\noindent\textbf{Constants}

	\vspace*{3pt}
	
	Constants like \(c,C,C',\dots\) are allowed to change value from line to line. Their dependency should be clear from the context. They never depend on the parameters that are allowed to vary in the course of a proof.
	
	\vspace*{3pt}
	
	\noindent\textbf{Distances, sets and boundary}
	
	\vspace*{3pt}
	
	 For \(a<b\in \R\), denote
	\begin{equation*}
		\llbracket a,b \rrbracket = \{n\in \Z:\ a\leq n\leq b\}.
	\end{equation*}
	
	Denote \(\sqcup\) for the union between disjoint sets. For \(x=(x_1,x_2)\in \R^2\), \(|x| = \sqrt{x_1^2+x_2^2}\) is the Euclidean norm, \(\rmd_2\) is the associated distance, \(\normsup{x} = \max(|x_1|,|x_2|) \) is the sup-norm of \(x\), and \(\rmd_{\infty}\) is the associated distance. Denote \(\rme_1 = (1,0)\), \(\rme_2 =(0,1)\).
	
	Denote \(\Lambda\Subset\Z^2\) for \(\Lambda\subset \Z^2\) finite and connected.
	Denote also \(\Lambda_n = \{-n,\dots, n\}^2\), and, for \(x\in \Z^2\), \(\Delta,\Lambda\subset \Z^2\),
	\begin{equation*}
		\Delta(x) = x+\Delta,\quad \Delta(\Lambda) = \bigcup_{x\in \Lambda} \Delta(x).
	\end{equation*}
	For \(\Lambda\Subset \Z^2\), define (see Figure~\ref{Fig:Lambda_and_inner_bnd})
	\begin{equation*}
		\partial^{\rmi} \Lambda = \{x\in \Lambda:\ \exists y\in \Lambda^c\ \norm{x-y}_{\infty} = 1 \},\quad
		\mathring{\Lambda} = \Lambda\setminus \partial^{\rmi}\Lambda.
	\end{equation*}Note that this is not the usual nearest-neighbour graph boundary. In the same spirit, let \(\bbL_K = ((2K+1)\Z)^2\), and for \(\Lambda\Subset \bbL_K\) define
	\begin{equation*}
		\partial^{\rmi} \Lambda = \{x\in \Lambda:\ \exists y\in \Lambda^c\ \norm{x-y}_{\infty} = 2K+1 \}.
	\end{equation*}
	
	\begin{figure}[h]
		\centering
		\includegraphics[scale=0.9]{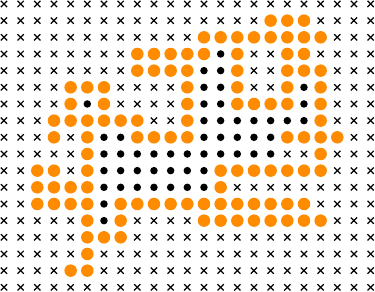}
		\caption{A volume \(\Lambda\) (circular dots), and its inner boundary \(\partial^{\rmi}\Lambda\) (orange dots). Crosses are \(\Z^2\setminus \Lambda\).}
		\label{Fig:Lambda_and_inner_bnd}
	\end{figure}
	
	\vspace*{3pt}
	
	\noindent\textbf{Graphs, paths}
	
	\vspace*{3pt}
	
	When not mentioned otherwise, \(\Z^2\) and its subsets are endowed with the nearest-neighbour (with respect to the Euclidean distance) graph structure, connectivity and paths are with respect to that graph structure (edges are of the form \(\{x,y\}\) with \(|x-y| = 1\)). The connectivity and paths induced by the graph structure given by the nearest-neighbour graph with respect to \(\rmd_{\infty}\) (edges are of the form \(\{x,y\}\) with \(\normsup{x-y} = 1\)) are called \emph{\(*\)-connectivity} and \emph{\(*\)-paths}.
	This notions are transported to \(\bbL_K\) via the natural bijection between \(\Z^2\) and \(\bbL_K\) (\(x\mapsto (2K+1) x\)).
	
	The notion of a \emph{separating set} will also be used. For \(A,B,D\subset \Lambda\Subset \Z^2\), let \(C_1,\dots,C_n\) be the connected components of \(\Lambda\setminus D\). Say that \(D\) \emph{separates \(A\) from \(B\)} if there are \(I,J\subset \{1,\dots,n\}\) such that \(I\cap J=\varnothing\), and \(A\subset\cup_{i\in I}C_i\), \(B\subset \cup_{i\in J}C_i\). The same notion transports to \(\bbL_K\).
	
	Finally, introduce \emph{surrounding sets}. Let \(G=(V,E)\) be an embedded (in a way that edges do not intersect or self-intersect) planar (multi)graph. See \(G\) as a subset of \(\R^2\) (sites are points, and edges are open continuous curves without intersections). Let \(D\subset V\cup E\). Let \(C_1,C_2,\dots\) be the finite connected components of \(\R^2\setminus D\). The set \(\cup_{k\geq 1} (C_k\cap(V\cup E))\) is the set of sites and edges \emph{surrounded by \(D\)}. For \(A\subset \R^2\) finite with \(A\subset \cup_k C_k\), the set \(\cup_{k\geq 1: C_k\cap A\neq \varnothing} (C_k\cap(V\cup E))\) is the set of sites and edges \emph{\(A\)-surrounded by \(D\)}.
	
	\begin{figure}[h]
		\centering
		\includegraphics{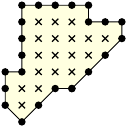}
		\hspace*{3cm}
		\includegraphics[scale=0.8]{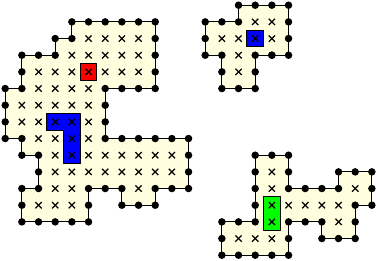}
		\caption{Left: a star-path (black disks) and the set it surrounds (yellow area, the surrounded sites of \(\Z^2\) are crosses). Right: a collection of paths (black disks) with the set they surround (in yellow); blue is separated form green by the paths, red is not separated from blue, and the set of blue-surrounded sites are the crosses inside the left and top right paths.}
		\label{Fig:paths_separation}
	\end{figure}
	
	\vspace*{3pt}
	
	\noindent\textbf{Scales}
	
	\vspace*{3pt}
	
	There will be three scales in the study: \(1\) is the scale of the model (unit lattice scale), \(l\) will be the ``Markov length scale'' of the model (this will be the scale at which the model ``behaves like'' a range one model), \(K\) will be the scale of the ``coarse-grained model'' that will be at the core of the analysis.
	
	\subsection{Models}
	\label{sec:models}
	
	\vspace*{3pt}
	
	\noindent\textbf{Families of measures}
	
	\vspace*{3pt}
	
	Let \(\Omega_x\), \(x\in \Z^2\), be finite or countable sets: the \emph{spin spaces}. For \(\Lambda\subset \Z^2\), let \(\Omega_{\Lambda} = \bigtimes_{x\in \Lambda}\Omega_x\), and let \(\calF_{\Lambda}\) be the \(\sigma\)-algebra generated by finitely supported events with support in \(\Lambda\). Let \(\Omega = \Omega_{\Z^2}\), \(\calF=\calF_{\Z^2}\).
	
	The objects under study are collections of probability measures \(\nu_{\Lambda}^t\), \(\Lambda\Subset \Z^2\), \(t\in T\), where \(T\) is an index set (one can think of \(t\) as boundary conditions). \(\nu_{\Lambda}^t\) is then a probability measure on \((\Omega_{\Lambda},\calF_{\Lambda})\). \(\sigma\) will denote a random variable of law \(\nu_{\Lambda}^t\). For \(\Delta\subset \Z^2\), \(\sigma_{\Delta}\) denotes the restriction of \(\sigma\) to \(\Delta\).
	
	For \(\Delta\subset \Lambda\Subset \Z^2\), \(t\in T\), and \(\xi\in \Omega_{\Lambda\setminus \Delta}\) such that \(\nu_{\Lambda}^t(\sigma_{\Lambda\setminus \Delta} =\xi)>0\), denote
	\begin{equation}
		\nu_{\Lambda,\Delta}^{t,\xi}(\cdot ) := \nu_{\Lambda}^t(\cdot \given \sigma_{\Lambda\setminus \Delta} =\xi).
	\end{equation}Also denote \(P|_{\Delta}\) for the restriction (projection) of the measure \(P\) to \(\calF_{\Delta}\).
	
	\begin{example}
		Classical examples to keep in mind are Gibbsian specifications where \(T= \Omega\) is the set of boundary conditions, or FK percolation (also with \(T=\Omega\)). But also FK percolation with modified coupling constants along the boundary of the system, in that case \(T\) is the set of pairs ``boundary conditions + values of the modified parameters''.
	\end{example}
	
	\vspace*{3pt}
	
	\noindent\textbf{Blocked model}
	
	\vspace*{3pt}
	
	A central notion is the one of \emph{blocked model}: for \(l\geq 0\), let
	\begin{equation*}
		l'=2l+1,\quad \bbL_l = (l'\Z)^2,\qquad \Lambda_l(x) = x+\{-l,\dots, l\}^2.
	\end{equation*}
	For \(\Lambda\subset \Z^2\), define the bijection
	\begin{equation*}
		\psi_{\Lambda}: \Omega_{\Lambda}\to \bigtimes_{u\in \bbL_l: \Lambda_l(u)\cap \Lambda\neq \varnothing} \Omega_{\Lambda_l(u)\cap \Lambda},
		\qquad
		\big(\psi_{\Lambda}(\omega)\big)_u = (\omega_{x})_{x\in \Lambda_l(u)\cap \Lambda}.
	\end{equation*}
	Slightly abusing notation, the subscript \(\Lambda\) will be dropped and should be clear from the context.
	
	\noindent In the same spirit, denote
	\begin{equation}
		X = \psi(\sigma),
	\end{equation}to mean ``if \(\sigma\sim \nu_{\Lambda}^t\), then \(X=\psi_{\Lambda}(\sigma)\)''.

	\vspace*{3pt}
	
	\noindent\textbf{Decoupling configurations}
	
	\vspace*{3pt}
	
	For \(\Lambda\Subset \Z^2,t\in T\), \(D\subset \Lambda\), \(\xi\in \Omega_{D}\), say that \(\xi\) is a \emph{\(D\)-decoupling configuration of \(\nu_{\Lambda}^t\)} if one has that, if \(\sigma\sim \nu_{\Lambda,\Lambda\setminus D}^{t,\xi}\), for any \(A,B\subset \Lambda\) such that \(D\) separates \(A\) from \(B\) in \(\Lambda\), \(\sigma_{A}\) and \(\sigma_B\) are independent.
	
	\subsection{Hypotheses}
	
	The hypotheses are divided into two sets: the first concern mixing (a generalisation of weak mixing), the second concern a weakening of Markov's property.
	
	\subsubsection*{Mixing Hypotheses}
	
	The first mixing hypotheses is the existence of a unique ``infinite volume measure'' compatible with the model under consideration. It is implied by the other mixing hypotheses, but it is nice to have it spelled separately as is gives a notation for a reference measure.
	\begin{enumerate}[label= (Mix\arabic*)]
		\renewcommand{\theenumi}{(Mix\arabic{enumi})}
		\item\label{hyp:Mix:inf_vol_meas} There is a (unique) probability measure \(\nu\) on \((\Omega,\calF)\) such that for any \(\Delta'\Subset \Z^2\), any \(A\in \calF_{\Delta'}\), any sequence \(\Delta_{k}\Subset\Z^2, k\geq 1\) such that \(\Delta_k\to\Z^2\), and any sequence \(t_k\in T, k\geq 1\),
		\begin{equation*}
			\lim_{k\to \infty} \nu_{\Delta_k}^{t_k}(A) = \nu(A).
		\end{equation*}
	\end{enumerate}
	The second mixing hypotheses is a technical condition which turns out to be a consequence of the ratio weak mixing property for rectangles.
	\begin{enumerate}[resume, label= (Mix\arabic*)]
		\item \label{hyp:Mix:exp_rel_density} \emph{Uniform Exponential Relaxation of Densities}: there are \(\Cmix\geq 0, \cmix>0\) such that, for any \(n,m,k>0\), any \(\Delta_{n,m}\Subset\Z^2\) with \(\Delta_{n,m} = x + \llbracket 0,n \rrbracket\times \llbracket 0,m \rrbracket\) for some \(x\in \Z^2\), 
		\begin{equation*}
			\sum_{\omega\in \Omega_{\Delta_{n,m}}} \inf_{\Delta\supset \Lambda_{k}(\Delta_{n,m})}\inf_{t,\xi} \nu_{\Delta,\Lambda_{k}(\Delta_{n,m})}^{t,\xi}\big(\sigma_{\Delta_{n,m}} = \omega\big) 
			\geq 1- \Cmix nm e^{-\cmix k},
		\end{equation*}where the inf over \(\xi\) is again over \(\xi\) such that \(\nu_{\Delta}^t(\sigma_{\Delta\setminus \Lambda_{k}(\Delta_{n,m})}= \xi)>0\).
	\end{enumerate}
	
	\subsubsection*{Markov-type Hypotheses}
	
	The first hypotheses is the existence of a ``conditional decoupling'' (conditional Markov) property, which is \emph{constituted of chains of local events}. One can have in mind having an open circuit in a annuli in planar FK percolation as example of such local piece of a decoupling event. These hypotheses take in a parameter: \(l\geq 0\) (as before, let \(l'=2l+1\)).
	\begin{enumerate}[label= (Mar\arabic*)]
		\renewcommand{\theenumi}{(Mar\arabic{enumi})}
		\item\label{hyp:Markov:Decoupling_circuits} There are events \(\MarkovSet_u\subset \Omega_{\Lambda_{l+l'}(u)}\), \(u\in \bbL_l\), events \(\MarkovSet_{u,\Lambda}^t\subset \Omega_{\Lambda_{l+l'}(u)\cap \Lambda}\), \(u\in \bbL_l\), \(\Lambda\Subset\Z^2\), \(t\in T\), such that for any \(\Lambda\Subset\Z^2\), \(t\in T\), and connected set \(\gamma\subset \{u\in \bbL_l:\ \Lambda_l(u)\cap \Lambda\neq \varnothing\}\) with \(\Lambda\setminus \Lambda_l(\gamma)\) not connected, one has that every \(\xi\in \Omega_{\Lambda_{l+l'}(\gamma) \cap \Lambda}\) such that
		\begin{equation*}
			\forall u\in \gamma,\quad \xi_{\Lambda_{l+l'}(u) \cap \Lambda} \in \begin{cases}
				\MarkovSet_{u,\Lambda}^t & \text{ if } \Lambda_{l+2l'}(u)\not\subset \Lambda,\\
				\MarkovSet_{u} & \text{ if } \Lambda_{l+2l'}(u)\subset \Lambda,
			\end{cases}
		\end{equation*}is a \(\Lambda_{l+l'}(\gamma)\cap \Lambda\)-decoupling configuration of \(\nu_{\Lambda}^t\).
	\end{enumerate}
	
	The second and third hypotheses quantify the probability of having a decoupling event. One has to think of \(\pbulk\) as being close to \(1\), while \(\theta>0\) can be as small as one wants. The idea behind the second hypotheses is that it allows to use the results of~\cite{Liggett+Schonmann+Stacey-1997} to say that decoupling regions (at a suitable scale) stochastically dominate a large density Bernoulli percolation.
	\begin{enumerate}[resume, label= (Mar\arabic*)]
		\item\label{hyp:Markov:local_Markov_bulk} One has
		\begin{equation*}
			\inf_{u\in \bbL_l}\inf_{\Lambda\supset \Lambda_{l+2l'}(u)}\inf_{t}\inf_{\xi\in \Omega_{\Lambda\setminus \Lambda_{l+l'}(u)}} \nu_{\Lambda,\Lambda_{l+l'}}^{t,\xi}( \sigma_{\Lambda_{l+l'}(u)}\in \MarkovSet_u ) \geq \pbulk,
		\end{equation*}where the inf over \(\xi\) is over \( \xi\) having positive \(\nu_{\Lambda}^t\)-probability, and \(\pbulk\geq 0\) is a control parameter.
	\end{enumerate}
	
	The last hypotheses allows the use of finite energy to perform local surgery. It is possible to relax this hypotheses a bit, but relaxing it makes everything even heavier and I did not find any useful application of a weaker (and sufficient) condition.
	\begin{enumerate}[resume, label= (Mar\arabic*)]
		\item\label{hyp:Markov:local_Markov_finite_energy} For every \(u\in \bbL_l\), there are \(\MarkovEl_u\in \Omega_{\Lambda_{l}(u)}\), \(\MarkovEl_{u,\Lambda}^t\in \Omega_{\Lambda_{l}(u)\cap \Lambda}\), \(\Lambda\Subset\Z^2\), \(t\in T\), such that:
		\begin{enumerate}
			\item for all \(u\in \bbL_l\), one has that if \(\vartheta\in \MarkovSet_{u}\), then \(\tilde{\vartheta} \in \MarkovSet_{u}\) where \(\tilde{\vartheta}\) is any configuration satisfying that for all \(v\in \bbL_l\) with \(\normsup{v-u} \leq l'\), \(\tilde{\vartheta}_{\Lambda_{l}(v)} \in \{\vartheta_{\Lambda_{l}(v)},\MarkovEl_v\}\);
			\item for all \(u\in \bbL_l\), one has that if \(\vartheta\in \MarkovSet_{u,\Lambda}^t\), then \(\tilde{\vartheta} \in \MarkovSet_{u,\Lambda}^t\) where \(\tilde{\vartheta}\) is any configuration satisfying that for all \(v\in \bbL_l\) with \(\normsup{v-u} \leq l'\), \(\tilde{\vartheta}_{\Lambda_l(v)\cap \Lambda} \in \{\vartheta_{\Lambda_l(v)\cap \Lambda},\MarkovEl_{v,\Lambda}^t\}\) if \(\Lambda_{l+l'}(v)\not\subset \Lambda\), and \(\tilde{\vartheta}_{\Lambda_l(v)} \in \{\vartheta_{\Lambda_l(v)\cap \Lambda}, \MarkovEl_v\}\) else;
			\item there is \(\theta>0\) such that
			\begin{gather*}
				\inf_{u}\inf_{\Lambda\supset \Lambda_{l+l'}(u)}\inf_t\inf_{\xi} \nu_{\Lambda,\Lambda_l(u)}^{t,\xi}(\sigma_{\Lambda_l(u)}= \MarkovEl_u) \geq \theta,
				\\
				\inf_{u}\inf_{\Lambda: \Lambda_{l+l'}(u)\cap \Lambda^c\neq \varnothing}\inf_t\inf_{\xi} \nu_{\Lambda,\Lambda_l(u)\cap \Lambda}^{t,\xi}(\sigma_{\Lambda_l(u)\cap \Lambda}= \MarkovEl_{u,\Lambda}^t) \geq \theta.
			\end{gather*}
		\end{enumerate}
	\end{enumerate}
	
	\begin{remark}
		\label{rem:MarkovEl_bnd_implies_MarSet_bnd}
		Note that the existence of the elements \(\MarkovEl_{x,\Lambda}^t,\MarkovEl_{x}\) in~\ref{hyp:Markov:local_Markov_finite_energy} implies the existence of the sets \(\MarkovSet_{x}, \MarkovSet_{x,\Lambda}^t\) (but not the lower bound of~\ref{hyp:Markov:local_Markov_bulk}).
	\end{remark}
	
	\begin{remark}
		If the model has the Markov property for the nearest-neighbour graph structure and a (weak form of) finite energy, Hypotheses~\ref{hyp:Markov:Decoupling_circuits},~\ref{hyp:Markov:local_Markov_bulk},
		and~\ref{hyp:Markov:local_Markov_finite_energy} are immediately satisfied for any \(l\geq 0\) (with \(\pbulk=1\)). The same is true for range-one models (Markov property for the \(*\)-graph structure).
	\end{remark}
	
	\subsection{Main Result}
	
	For \(\ell\geq 1\), let \(\ell'=2\ell +1\), and denote \(\bbL_{\ell} = (\ell'\Z)^2\) so that \(\Z^2 = \bigsqcup_{i\in \bbL_{\ell}} \Lambda_{\ell}(i)\). For \(\bar{p}\in [0,1]^{\bbL_{\ell}}\), \(\Lambda\subset \Z^2\), denote \(\omega\sim P_{\Lambda;\ell,\bar{p}}\) the site percolation model on \(\Z^2\) obtained by taking \(X_i,i\in \bbL_{\ell}\) an independent family of random variables with \(X_i\sim \mathrm{Bern}(p_i)\) and setting
	\begin{equation*}
		\omega(x) = \begin{cases}
			0 & \text{ if } x\in \Lambda^c,\\
			X_i & \text{ if } x\in \Lambda \cap \Lambda_{\ell}(i).
		\end{cases}
	\end{equation*}
	
	The main result of this paper is the following statement.
	\begin{theorem}
		\label{thm:main}
		Let \((\nu_{\Lambda}^t)_{\Lambda,t}\) be as described in Section~\ref{sec:models}. Suppose that the mixing Hypotheses~\ref{hyp:Mix:inf_vol_meas}, and~\ref{hyp:Mix:exp_rel_density} with \(\Cmix\geq 0, \cmix >0\) hold. There is \(p_0=p_0(\cmix)\in (0,1)\) such that if, for some \(l\geq 0\), Hypotheses~\ref{hyp:Markov:Decoupling_circuits},~\ref{hyp:Markov:local_Markov_bulk} with \(\pbulk\geq p_0\), and~\ref{hyp:Markov:local_Markov_finite_energy} with \(\theta>0\) all hold, then, for any \(p\in (0,1)\), there is \(\ell\geq 1\), and \(q\in (0,1)\) such that for any \(\Lambda\Subset \Z^2\), \(t\in T\), \(\Delta_1,\Delta_2\subset \Lambda\), and \(\xi,\xi'\in \Omega_{\Delta_1}\) with \(\nu_{\Lambda}^t(\sigma_{\Delta_2} = \xi)\nu_{\Lambda}^t(\sigma_{\Delta_2} = \xi')>0\),
		\begin{equation*}
			\dTV\big(\nu_{\Lambda}^t(\sigma_{\Delta_1}\in \cdot \given \sigma_{\Delta_2} = \xi),\nu_{\Lambda}^t(\sigma_{\Delta_1}\in \cdot \given \sigma_{\Delta_2} = \xi')\big) \leq P_{\Lambda;\ell,\bar{p}}\big( \Delta_2 \leftrightarrow_* \Delta_1\big),
		\end{equation*}where \(\leftrightarrow_*\) means \(*\)-connected, and \(\bar{p}\) is given by
		\begin{equation*}
			p_i = p_i(\Lambda,\Delta_2) = \begin{cases}
				1 & \text{ if } \Lambda_{\ell}(i) \cap \Delta_2 \neq \varnothing,\\
				p & \text{ if } \Lambda_{\ell+\ell'}(i) \subset  \Lambda\setminus\Delta_2,\\
				0 & \text{ if } \Lambda_{\ell}(i) \subset \Lambda^c,\\
				q & \text{ else }.
			\end{cases}
		\end{equation*}
	\end{theorem}
	
	This result is obtained from a slightly finer upper bound, but I do not believe that the finer result has a real added value compare to this one, and this statement is easier to digest. I also believe that a higher dimensional version of this result should hold (indeed: there the lack of strong mixing is manifested by the percolation of information along the boundary, which potentially has arbitrarily large percolation parameter).
	
	Strong mixing then follows for nice enough volumes: the percolation model \(P_{\Lambda;\ell,\bar{p}}\) is very subcritical in the bulk and, when the system has a ``nicely one-dimensional boundary'', the inhomogeneities close to the boundary cannot destroy this decay. In the applications, the case of large squares is considered as it is easy to show exponential decay there (see Lemma~\ref{lem:exp_dec_inhomo_bnd}). But the result also holds with more complicated volumes. As far as applying the result is concerned, I believe Theorem~\ref{thm:main} to be more convenient than putting some a priori restriction on the class of volume considered, and stating exponential decay for that class. Volumes under consideration might vary a great deal depending on what one is trying to do. Compared to the class of volumes considered in~\cite{Alexander-2004}, which give strong mixing for simply connected volumes in the case of planar FK percolation, the present criterion is a priori neither weaker nor stronger: indeed, it is easy to find not simply connected regions for with it is easy to establish exponential decay under \(P_{\Lambda;\ell,\bar{p}}\), and therefore apply Theorem~\ref{thm:main}, while it is not clear (and even wrong!) that \(P_{\Lambda;\ell,\bar{p}}\) has exponential decay of connectivities uniformly over simply connected \(\Lambda\). For the reader interested in less simple applications of Theorem~\ref{thm:main}, the arguments of~\cite[Section 2]{Alexander-2004} permit to study exponential decay of connections probability under \(P_{\Lambda;\ell,\bar{p}}\) for various family of volumes.
	
	\subsection{Relaxing the mixing hypotheses}
	
	Using the results of~\cite{Alexander-1998}, one can relax Hypotheses~\ref{hyp:Mix:exp_rel_density}. This is the content of the next Theorem. The result imported from~\cite{Alexander-1998} (which is 99\% of the proof of the next Theorem) is summarized in Lemma~\ref{lem:app:mixing_to_ratioMixing} (more precisely, Lemma~\ref{lem:app:mixing_to_ratioMixing} is an abstract formulation of part of~\cite[Section 5]{Alexander-1998}, and the rest of the proof is a simplification of an argument of~\cite[Section 5]{Alexander-1998}).
	\begin{theorem}
		\label{thm:weakMix_to_ratioWeakMix}
		Suppose there are \(C_*\geq 0,c_*>0\) such that for any \(\Delta\subset \Delta'\Subset\Z^2\), any \(\Delta'\subset\Lambda_1,\Lambda_2\Subset \Z^2\), and any \(\xi_i\in \Omega_{\Lambda_i\setminus\Delta'}\) with \(\nu_{\Lambda_i}^{t_i}(X_{\Lambda_i\setminus \Delta'} = \xi_i)>0\),
		\begin{equation}
			\label{eq:weak_mix}
			\rmd_{\mathrm{TV}}\big(\nu_{\Lambda_1,\Delta'}^{t_1,\xi_1}|_{\Delta}, \nu_{\Lambda_2,\Delta'}^{t_2,\xi_2}|_{\Delta}\big) \leq C_*\sum_{x\in \Delta}\sum_{y\in (\Delta')^c}e^{-c_*\normsup{x-y}}.
		\end{equation}Suppose moreover that Hypotheses~\ref{hyp:Markov:Decoupling_circuits}, and~\ref{hyp:Markov:local_Markov_bulk} with \(\pbulk\) close enough to \(1\) hold (independently of the value of \(C_*\geq 0,c_*>0\)). Then, Hypotheses~\ref{hyp:Mix:inf_vol_meas} and~\ref{hyp:Mix:exp_rel_density} hold.
	\end{theorem}
	\begin{remark}
		The key difference between~\ref{hyp:Mix:exp_rel_density} and the Hypotheses of Theorem~\ref{thm:weakMix_to_ratioWeakMix} is that in the second case one has to deal with \emph{differences of probability measures} whilst the first is about \emph{ratio between ``densities''}.
	\end{remark}
	\begin{proof}[Proof of Theorem~\ref{thm:weakMix_to_ratioWeakMix}]
		Let \(n,m,k>0\). Note that it is sufficient to prove the statement for \(k\geq C\ln(nm)\), \(nm\geq n_0\) for \(C,n_0\) large as one can adapt the value of \(\Cmix\) to fit smaller values of \(k,nm\). Suppose this from now on. Let \(v\in \Z^2\), and set \(\Delta = v + \llbracket 0, n \rrbracket\times \llbracket 0,m \rrbracket\), \(\Delta' = v+ \llbracket -k,n + k \rrbracket\times \llbracket -k,m + k \rrbracket\). Fix some \(t_0\in T\) and let \(\mu \equiv \nu_{\Delta'}^{t_0}\) be the reference measure. Let then \(\Lambda\supset \Delta'\), \(t\in T\), and \(\xi\in \Omega_{\Lambda\setminus\Delta'}\) with \(\nu_{\Lambda}^t(\sigma_{\Lambda\setminus \Delta'} = \xi)>0\). Let
		\begin{equation*}
			W = \Big\{ x:\ \rmd_{\infty}\big(x,(\Delta')^c\big)\geq k/3,\ \rmd_{\infty}\big(x,\Delta\big) 		\geq k/3 \Big\}.
		\end{equation*}
		The goal is to apply Lemma~\ref{lem:app:mixing_to_ratioMixing} to the measures \(\mu|_{W\sqcup \Delta}\) (with marginals \(\mu|_{\Delta}\) and \(\mu|_{W}\)) and \(\nu_{\Lambda,\Delta'}^{t,\xi}|_{W\sqcup \Delta}\) (with marginals \(\nu_{\Lambda,\Delta'}^{t,\xi}|_{\Delta}\) and \(\nu_{\Lambda,\Delta'}^{t,\xi}|_{W}\)).
		
		Let \(l\geq 0\) be the parameter linked to~\ref{hyp:Markov:Decoupling_circuits}, and~\ref{hyp:Markov:local_Markov_bulk}. Let \(D\) be the event that there is a simple closed path of blocks, \(\gamma\subset \bbL_l\), with \(\Lambda_{l+l'}(\gamma)\subset W\) surrounding \(\Delta\) such that \(\sigma_{\Lambda_{l+l'}(u)}\in \MarkovSet_u\) for all \(u\in \gamma\). As \(\pbulk\) is taken large enough, the set
		\begin{equation*}
			\big\{u\in \bbL_l:\ \Lambda_{l+l'}(u)\subset W,\ \sigma_{\Lambda_{l+l'}}\in \MarkovSet_u \big\}
		\end{equation*}stochastically dominates a Bernoulli percolation of large parameter for any of the measures under consideration (by the results of~\cite{Liggett+Schonmann+Stacey-1997}), so that \(D\) has \(\mu|_{W}\)- and \(\nu_{\Lambda,\Delta'}^{t,\xi}|_W\)-probability at least \(1-Cnme^{-ck}\) for some constants \(c>0,C\geq 0\) depending only on \(\pbulk\) (and \(l\)). Hypotheses~\ref{hyp:Markov:Decoupling_circuits} then gives the third condition of Lemma~\ref{lem:app:mixing_to_ratioMixing} with \(\epsilon_3= Cnme^{-ck}\).
		
		Now,~\eqref{eq:weak_mix} implies that
		\begin{equation*}
			\rmd_{\mathrm{TV}}\big(\nu_{\Lambda, \Delta'}^{t,\xi}|_{W}, \mu|_W\big) \leq Cnm e^{-ck},
		\end{equation*}for some \(c>0,C\geq 0\) depending only on \(C_*,c_*\) (recall the constraints on \(nm,k\)). This gives the second condition with \(\epsilon_2 = Cnme^{-ck}\). Finally, an additional application of~\eqref{eq:weak_mix} gives the first condition with \(\epsilon_1 = Cnme^{-ck}\) for some \(c>0,C\geq 0\) depending on \(c_*,C_*\). Lemma~\ref{lem:app:mixing_to_ratioMixing} then gives that for any \(\omega\in \Omega_{\Delta}\) with \(\mu(\sigma_{\Delta} = \omega)>0\),
		\begin{equation*}
			\frac{\nu_{\Lambda,\Delta'}^{t,\xi}(\sigma_{\Delta} = \omega)}{\mu(\sigma_{\Delta} = \omega)} \geq 1-Cnm e^{-ck}.
		\end{equation*}These bounds are uniform over \(\xi,\Lambda\), so one obtains that
		\begin{equation*}
			\sum_{\omega\in \Omega_{\Delta}} \inf_{t, \Lambda\supset \Delta'}\inf_{\xi} \nu_{\Lambda,\Delta'}^{t,\xi} \big(\sigma_{\Delta} = \omega\big) 
			\geq (1-Cnme^{-ck}) \sum_{\omega\in \Omega_{\Delta}} \mu(\sigma_{\Delta} = \omega) = 1-Cnme^{-ck},
		\end{equation*}which is~\ref{hyp:Mix:exp_rel_density}.
	\end{proof}
	
	\section{Proof of Theorem~\ref{thm:main}}
	
	\subsection{Mechanism behind the proof}
	
	To get the idea, the proof is described here for a Markovian model (Markovian with respect to the square lattice with n.n. edges, so \(l=0\), \(\pbulk = 1\)).
	The idea of the proof is to construct a coupling between \(\nu_{\Lambda,\Lambda\setminus \Delta_2}^{t,\xi}|_{\Delta_1}\) for different values of \(\xi\). This is done by sampling from \(\nu_{\Lambda,\Lambda\setminus \Delta_2}^{t,\xi}\) using an exploration sampling (exploration sampling are reviewed in Section~\ref{sec:prf_of_main:patch_expl}). The exploration goes as follows: it explores new parts of space by first exploring large blocks of sites away from the already explored region, and then explores paths joining the already explored region and the new block.
	
	The blocks being far from what was explored before, the sampling of the new block is done uniformly over the past (and thus over the value of \(\xi\)) with high probability. In this case, the block is called good. Then, the patching between the new good block and previously explored good blocs is done by exploring paths between the different blocks. The sampling on a given path is done uniformly over the past with positive probability by finite energy, and the procedure is tailored so that there are sufficiently many paths between blocks in order for the entropy of paths to compensate the finite energy cost. In particular, one can find a path joining good blocks which is sampled uniformly over the past (and thus \(\xi\)), called a good path, with high probability.
	
	Blocks and paths close to the boundary of the system are handled only using finite energy (but the boundary being one-dimensional, this is sufficient).
	
	Once a ``circuit of good blocks joined by good paths'' surrounding \(\Delta_2\) is explored, it decouples its inside from its outside and the sampling of its exterior is then uniform over \(\xi\). Figure~\ref{Fig:mechanism_prf} gives a graphical idea of how the exploration is done.
	
	\begin{figure}[h]
		\centering
		\includegraphics[scale=0.8]{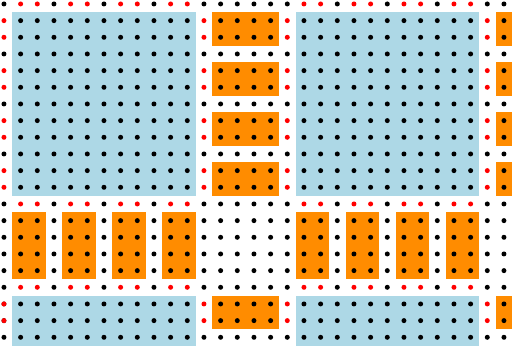}
		\caption{Blue blocks are explored first, then orange and finally red sites. Weak mixing is used to show that sampling blue and orange blocks away from the system boundary is done uniformly over the past with high probability. The entropy of the red-orange junctions between blue blocks compensates for the use of finite energy on the red sites.}
		\label{Fig:mechanism_prf}
	\end{figure}
	
	Of course, the setup of the present paper being a bit less nice than Markovian models, everything is a bit more technical to write down, and the final exploration turns out to be rather difficult to digest. Apologies to the reader for this, I did not find a better way of writing the proof (which most certainly exists, given the simplicity of the underlying mechanism).
	
	\subsection{Patch exploration}
	\label{sec:prf_of_main:patch_expl}
	
	\begin{definition}
		\label{def:patchExploration_sampling}
		Let \((\Omega,\calF,P)\) be a probability space. Let \(I\) be a finite set, and \((\calX_x)_{x\in I}\) a collection of finite sets equipped with the discrete \(\sigma\)-algebra: \(\calF_x=\calP(\calX_x)\). Set \(\calX = \bigtimes_{x\in I}\calX_x\), \(\calF= \bigotimes_{x\in I}\calF_x\), and let \(X\) be a \(\calX\)-valued random variable. Denote \(X_{J}\) the restriction (projection) of \(X\) to \(\bigtimes_{x\in J}\calX_x\). Suppose that for every \(J\subset I\), and for \(P\)-almost every realization of \(X_{J}\), \(\xi\), the conditional measure \(P(\cdot \given X_{J}=\xi)\) is well defined. A \emph{patch exploration sampling} for \(X\) is a the data of
		\begin{enumerate}
			\item an i.i.d. sequence \(U_0,U_1,U_2,\dots\) of uniform random variables on \([0,1)\);
			\item a collection of measurable functions: \(\calA_0 = \varnothing\), and for \(k\geq 1\), \(\calA_{k}: [0,1)^{k}\to \{ J\subset I\}\) such that
			\begin{itemize}
				\item \(\calA_{k}(u_0,\dots, u_{k-1})\subset \calA_{k+1}(u_0,\dots, u_{k})\) for all \(k\)'s, \(u_0,\dots,u_k\in [0,1)\),
				\item there is \(T_*\in \N\) such that for any \(u_0,u_1,\dots\)
				\begin{equation*}
					\calA_{k}(u_0,\dots, u_{k-1}) = I \ \forall k\geq T_*;
				\end{equation*}
			\end{itemize}
			\item a collection of measurable functions: for \(\Lambda\subset I\), \(J\subset I\setminus \Lambda\), and \(\xi\in \calX_{\Lambda}\) such that \(P(\cdot \given X_{\Lambda}=\xi)\) is well defined,
			\begin{equation*}
				F_{\Lambda;J}^{\xi} : [0,1) \to \calX_J,
			\end{equation*}is such that the law of \(F_{\Lambda;J}^{\xi}(U_0)\) is \(P\big( X_J\in \cdot \given X_{\Lambda}=\xi\big)\).
		\end{enumerate}
	\end{definition}
	
	We then have the next simple fact.
	\begin{lemma}
		\label{lem:patchExploration_sampling}
		With the notations and hypotheses of definition~\ref{def:patchExploration_sampling}, let
		\begin{equation*}
			\calA_{0}= \varnothing,\qquad \calA_{k} = \calA_{k}(U_0,\dots, U_{k-1}),\quad \calV_{k} = \calA_{k}\setminus \calA_{k-1},
		\end{equation*}and define inductively \(Y\in \calX\) by
		\begin{equation*}
			Y_{\calV_k} = Y_{\calV_k}(U_k) = F_{\calA_{k-1}; \calV_k}^{Y_{\calA_{k-1}}}(U_k).
		\end{equation*}Then,
		\begin{enumerate}
			\item \(Y\) is well defined as \(\calA_{k}\) stabilises, and \(Y \lawEq X\),
			\item for every \(\bbN\)-valued random variable \(T\) with \(\{T=n\}\in \sigma(U_0,\dots, U_{n-1})\) for \(n\geq 1\), and \(T\leq T_*\) a.s.,
			\begin{equation*}
				P\big(Y_{I\setminus \calA_{T}} \in \cdot \given U_0,\dots, U_{T-1}\big) \asEq P\big(X_{I\setminus \calA_T} \in \cdot \given X_{\calA_T} = Y_{\calA_T}\big).
			\end{equation*}
		\end{enumerate}
	\end{lemma}
	\begin{proof}
		Start by a version of the second point at deterministic times.
		\begin{claim}
			For any \(n \geq 0\),
			\begin{equation*}
				P\big(Y_{I\setminus \calA_{n}} \in \cdot \given U_0,\dots, U_n\big) \asEq P\big(X_{I\setminus \calA_n} \in \cdot \given X_{\calA_n} = Y_{\calA_n}\big).
			\end{equation*}
		\end{claim}
		\begin{proof}
			Denote \(\calA_k^c \equiv I\setminus \calA_k\). Proceed by reverse induction: as \(I\setminus \calA_{n} = \varnothing\) for \(n\geq T_*\), the result is direct for those \(n\). If the result holds true for \(n+1\), one has that for any \(x\in \calX\),
			\begin{align*}
				&P\big(Y_{\calA_{n}^c} = x_{\calA_{n}^c} \given U_0,\dots, U_{n-1}\big)
				\\
				&\quad=
				\sum_{V\subset I} E_{n}\big( \mathds{1}_{\calV_{n+1}= V} P\big(Y_{\calA_{n+1}^c}= x_{\calA_{n+1}^c}, Y_{V}=x_V \given U_0,\dots, U_n\big)\big)
				\\
				&\quad=
				\sum_{V\subset I} E_{n}\big( \mathds{1}_{\calV_{n+1}= V} \mathds{1}_{F_{\calA_n;V}^{Y_{\calA_{n}}}(U_{n+1}) = x_{V} } P\big(Y_{\calA_{n+1}^c}= x_{\calA_{n+1}^c} \given U_0,\dots, U_n\big)\big)
				\\
				&\quad=
				\sum_{V\subset I} P(X_{V}=x_V\given X_{\calA_{n}} = Y_{\calA_n} )E_{n}\big( \mathds{1}_{\calV_{n+1}= V} P\big(Y_{\calA_{n+1}^c}= x_{\calA_{n+1}^c} \given U_0,\dots, U_n\big)\big)
				\\
				&\quad=
				\sum_{V\subset I} P(X_{V}=x_V\given X_{\calA_{n}} = Y_{\calA_n} )E_{n}\big( \mathds{1}_{\calV_{n+1}= V} P\big(X_{\calA_{n+1}^c} = x_{\calA_{n+1}^c} \given X_{\calA_{n+1}} = Y_{\calA_{n+1}} \big)\big)
				\\
				&\quad=
				\sum_{V\subset I} E_{n}\big( \mathds{1}_{\calV_{n+1}= V} P\big(X_{\calA_{n}^c} = x_{\calA_{n}^c} \given X_{\calA_{n}} = Y_{\calA_{n}} \big)\big) = P\big(X_{\calA_{n}^c} = x_{\calA_{n}^c} \given X_{\calA_{n}} = Y_{\calA_{n}} \big),
			\end{align*}where \(E_{n}\) denotes expectation with respect to \(U_{n},U_{n+1},\dots\), and we used the independence between the \(U_k\)'s and the properties of the \(f\)'s in the third equality, and the fact that the result for \(n+1\) was supposed true in the fourth. This implies that if the result holds for \(n+1\) it holds for \(n\), which concludes the proof of the Claim.
		\end{proof}
		Applying the Claim for \(n=0\) gives the first point of the Lemma. The second point is a form of strong Markov property and follows directly from the Claim and the fact that \(\{T=n\}\) is measurable with respect to \(U_0,\dots, U_{n-1}\).
	\end{proof}
	
	\subsection{Coarse-Graining and rescaling}
	\label{sec:prf_main:CG}
	
	Let \(a,r,L\) be positive integers. Let
	\begin{gather*}
		L_1 = r^L 3L,\quad K = L_1 + (a+1)L,\quad K'=(2K+1),\\
		\rmB = \Lambda_{L_1},\quad \rmB' = \Lambda_{aL}(\rmB).
	\end{gather*}
	Let \(l\geq 0\) be the parameter linked with the ``Markov-type hypotheses''. It will be more convenient to work with the blocked field \(X = \psi(\sigma)\). For the rest of the section, rescale everything by \(2l+1\), and identify \(\bbL_l\) with \(\Z^2\), and \(\Z^2\) with \((\delta_l \Z)^2\), where \(\delta_l = \tfrac{1}{2l+1}\). In particular, \(\sigma\) lives on \((\delta_l\Z)^2\) and \(X\) lives on \(\Z^2\). for \(\Lambda\subset (\delta_l\Z)^2\), introduce
	\begin{equation}
		\label{eq:def:micro_to_lattice}
		\calS(\Lambda) = \{x\in \Z^2:\ (x+[-l,l])\cap \Lambda\neq \varnothing\}.
	\end{equation}Also define the ``reverse'' mapping:
	\begin{equation}
		\label{eq:def:lattice_to_micro}
		\calR(\Lambda) = (\delta_{l}\Z)^2\cap \bigcup_{u\in \Lambda}(u+[-\tfrac{1}{2}, \tfrac{1}{2}]).
	\end{equation}
	
	\subsubsection*{Coarse-Graining of infinite volume}
	
	Introduce then
	\begin{equation*}
		\Gamma = \bbL_K = (K'\Z)^2,\quad E_{\Gamma} = \big\{\{i,j\}\subset \Gamma:\ \norm{i-j}_2 = K' \big\}.
	\end{equation*}Note that if \(\{i,j\}\in E_{\Gamma}\), then \((i-j)/K'\in \{\pm\rme_1,\pm\rme_2\}\). For \(\{i,j\}\in E_{\Gamma}\), define
	\begin{equation*}
		\rme_{ij} = \frac{j-i}{K'},\quad \rme_{ij}^{\perp}\in \{\rme_1,\rme_2\} \text{ is such that } \rme_{ij}^{\perp}\cdot \rme_{ij} = 0,
	\end{equation*}and let \(\Theta_{ij}\) be the linear transformation such that
	\begin{equation*}
		\Theta_{ij} \rme_1 = \rme_{ij}, \quad \Theta_{ij} \rme_2 = \rme_{ij}^{\perp}.
	\end{equation*} Note that they are related via: \(\Theta_{ij}\) is the composition of \(\Theta_{ji}\) with the mirror symmetry sending \(\rme_{ij}\) to \(-\rme_{ij} = \rme_{ji}\). With these notations, and for \(k\in \{\pm 1,\dots, \pm r^L\}\) let (see Figure~\ref{Fig:coarse_graining1} and Figure~\ref{Fig:coarse_graining2})
	\begin{gather*}
		E_{\refBloc}^k = \begin{cases}
			\llbracket L_1+L+1, L_1 + (2a+1)L \rrbracket \times \llbracket (3k-2)L+1, (3k-1)L\rrbracket & \text{ if } k>0\\
			\llbracket L_1+L+1, L_1 + (2a+1)L \rrbracket \times \llbracket (3k+1)L, (3k+2)L-1\rrbracket & \text{ if } k<0
		\end{cases},
		\\
		F_{\refBloc}^k = \begin{cases}
			\llbracket L_1+1, L_1 + (2a+2)L \rrbracket \times \llbracket (3k-2)L+1, (3k-1)L\rrbracket & \text{ if } k>0\\
			\llbracket L_1+1, L_1 + (2a+2)L \rrbracket \times \llbracket (3k+1)L, (3k+2)L-1\rrbracket & \text{ if } k<0
		\end{cases},
		\\
		H_{\refBloc} = \llbracket L_1 +1 , L_1 + (2a+2)L\rrbracket \times \llbracket -L_1,L_1\rrbracket,
	\end{gather*}be the reference ``edge-blocks'', and
	\begin{gather*}
		E_{ij}^k = i + \Theta_{ij} E_{\refBloc}^k = j + \Theta_{ji} E_{\refBloc}^k = E_{ji}^k,\\
		F_{ij}^k = i + \Theta_{ij} F_{\refBloc}^k = j + \Theta_{ji} F_{\refBloc}^k = F_{ji}^k,\\
		H_{ij} = i + \Theta_{ij} H_{\refBloc} = j + \Theta_{ji} H_{\refBloc} = H_{ji},
	\end{gather*}be the ``edge-associated-blocks'' of the coarse-graining. Note that all those sets depend only on the edge \(\{i,j\}\) and not on the particular orientation of this edge. The notation \(E_{e}^k, F_{e}^k, H_{e}\) with \(e\in E_{\Gamma}\) is thus unambiguously defined. The sets \(\rmB(i)\)'s will be referred to as ``site-blocks'', the sets \(H_e\)'s as ``edge-blocks'', and the sets \(E_{e}^k\) as ``micro-edge-blocks''. An important point of these definitions is that for \(k\neq k'\),
	\begin{equation*}
		\Lambda_{L}(E_{e}^k)\cap \Lambda_{L}(E_{e}^{k'}) = \varnothing,\qquad \Lambda_{L}(E_{e}^k)\cap \rmB(i) = \varnothing.
	\end{equation*}where \(e=\{i,j\}\). I.e.: the different regions of the coarse-graining are well separated relatively to their sizes.
	
	\begin{figure}[h]
		\centering
		\includegraphics[scale=0.5]{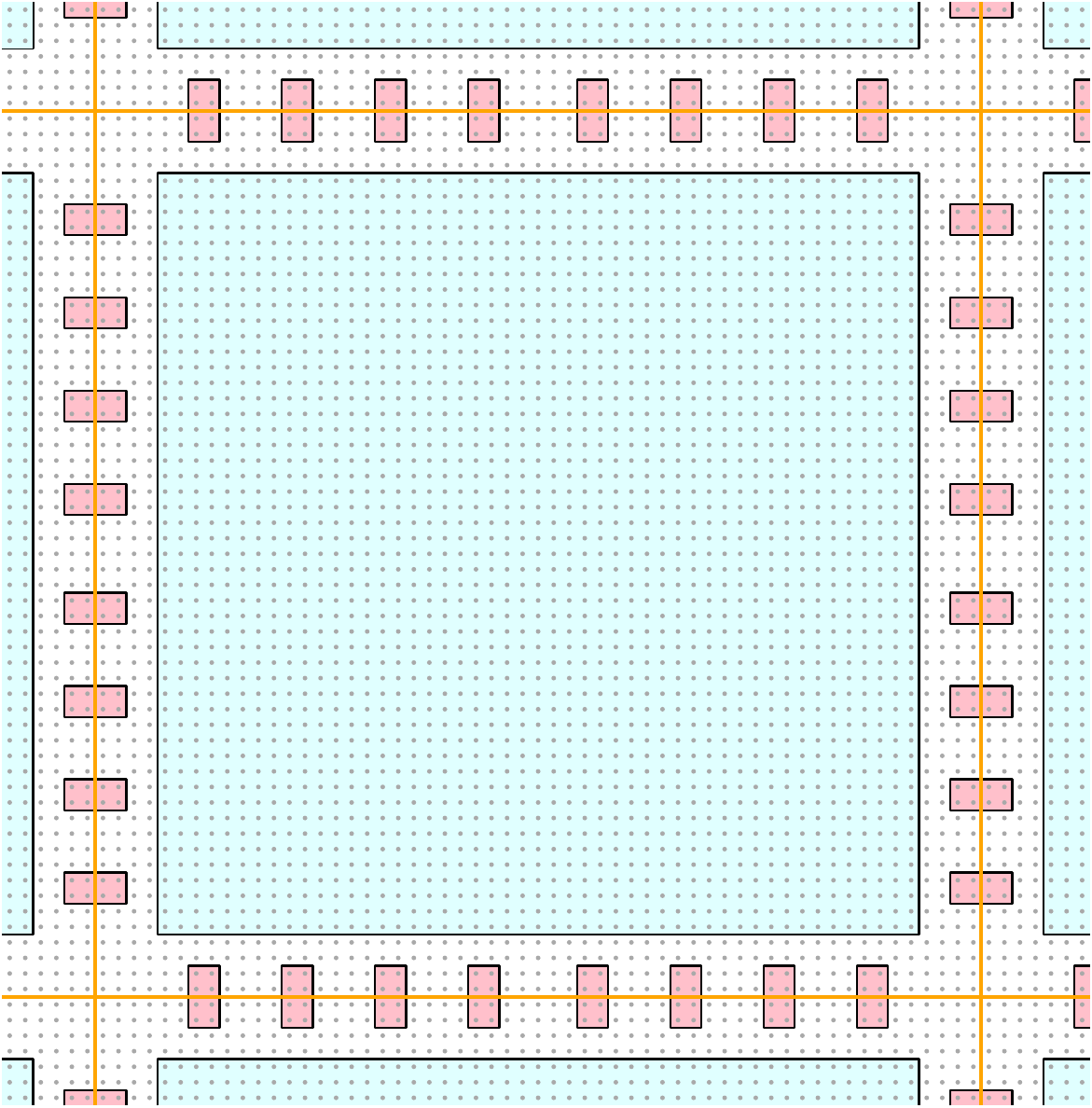}
		\caption{Illustration of the partitioning with \(r=L=2\), \(a=1\). The ``elementary cell'' of the partitioning is depicted in orange. The sets \(E_{ij}^k\) correspond to the rectangular pink boxes, the sets \(F_{ij}^k\) to rectangular boxes plus the sites between the rectangular box and the neighbouring square boxes, and the sets \(H_{ij}\) are the sets of sites between two square boxes. Grey dots represent sites where \(X=\psi(\sigma)\) lives.}
		\label{Fig:coarse_graining1}
	\end{figure}
	
	Finally, for \(i,j\in \Gamma\) with \(e=\{i,j\}\in E_{\Gamma}\), and \(k\in \{\pm 1, \dots, \pm r^L\}\), define (see Figure~\ref{Fig:coarse_graining2}):
	\begin{equation}
		\label{eq:def:sides_micro_edges}
		\begin{gathered}
			\varrho (e,i,k) = \big\{x\in E_{e}^k:\ \rmd_{\infty}(x,\rmB(i)) = L+1\big\},
			\\
			\varrho(i,j,k) = \big\{x\in \rmB(i):\ \rmd_{\infty}(x,E_{ij}^k) = L+1\big\}.
		\end{gathered}
	\end{equation}For any \(i,j\in \Gamma\), \(k\in \{\pm 1, \dots, \pm r^L\}\), and any \(x\in \varrho (e,i,k)\), \(y\in \varrho(i,j,k)\), let \(\gamma_{e,k,x,y} \equiv \gamma_{e,k,y,x}\) be a fixed self-avoiding path of length less or equal to \(2L\) connecting \(x\) to \(y\) with \(\gamma_{e,k,x,y}\) included in \(F_e^k\setminus E_{e}^k\) (except for its endpoints).
	
	\begin{figure}[ht!]
		\centering
		\includegraphics[scale=0.8]{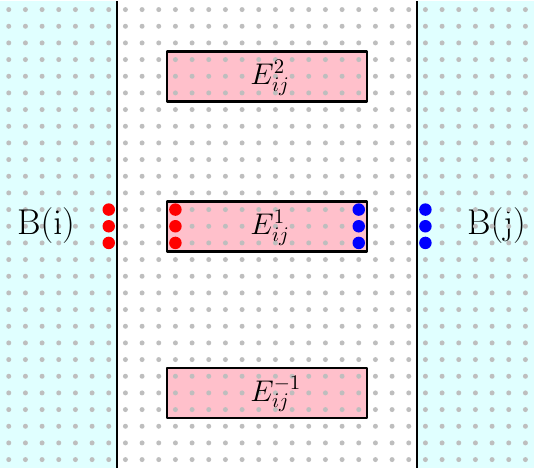}
		\caption{The set \(\varrho(\{i,j\},i,1)\) is given by the red sites in \(E_{ij}^1\) while the set \(\varrho(\{i,j\}, j,1)\) is given by the blue sites in \(E_{ij}^1\). The sets \(\varrho(i,j,1)\) and \(\varrho(j,i,1)\) are respectively given by the red sites in \(\rmB(i)\) and the blue sites in \(\rmB(j)\). Here, \(L=3, a=2\). Grey dots represent sites where \(X=\psi(\sigma)\) lives.}
		\label{Fig:coarse_graining2}
	\end{figure}
	
	\subsubsection*{Coarse-Graining of finite volumes}
	
	The goal is to study measures in finite volume with boundary conditions. So, introduce a coarse-grained version of finite volumes (see Figure~\ref{Fig:Lambda_CG}): for \(\Lambda\Subset (\delta_l\Z)^2\), define
	\begin{gather*}
		[\Lambda] = \{i\in \Gamma:\ \calR(\Lambda_{K}(i))\cap \Lambda\neq \varnothing\},
		\quad
		[\Lambda]_{\bulk} = \{i\in [\Lambda]:\ \calR(\Lambda_{K+2}(i))\subset \Lambda\},
		\\
		[\Lambda]_{\bnd} = [\Lambda]\setminus [\Lambda]_{\bulk},
		\quad
		[\Lambda]_{\ext} = \Gamma\setminus [\Lambda].
	\end{gather*}
	Note that
	\begin{equation*}
		\Gamma = [\Lambda]_{\bnd} \sqcup [\Lambda]_{\bulk} \sqcup [\Lambda]_{\ext}.
	\end{equation*}

	\begin{figure}[ht!]
		\centering
		\includegraphics[scale=0.7]{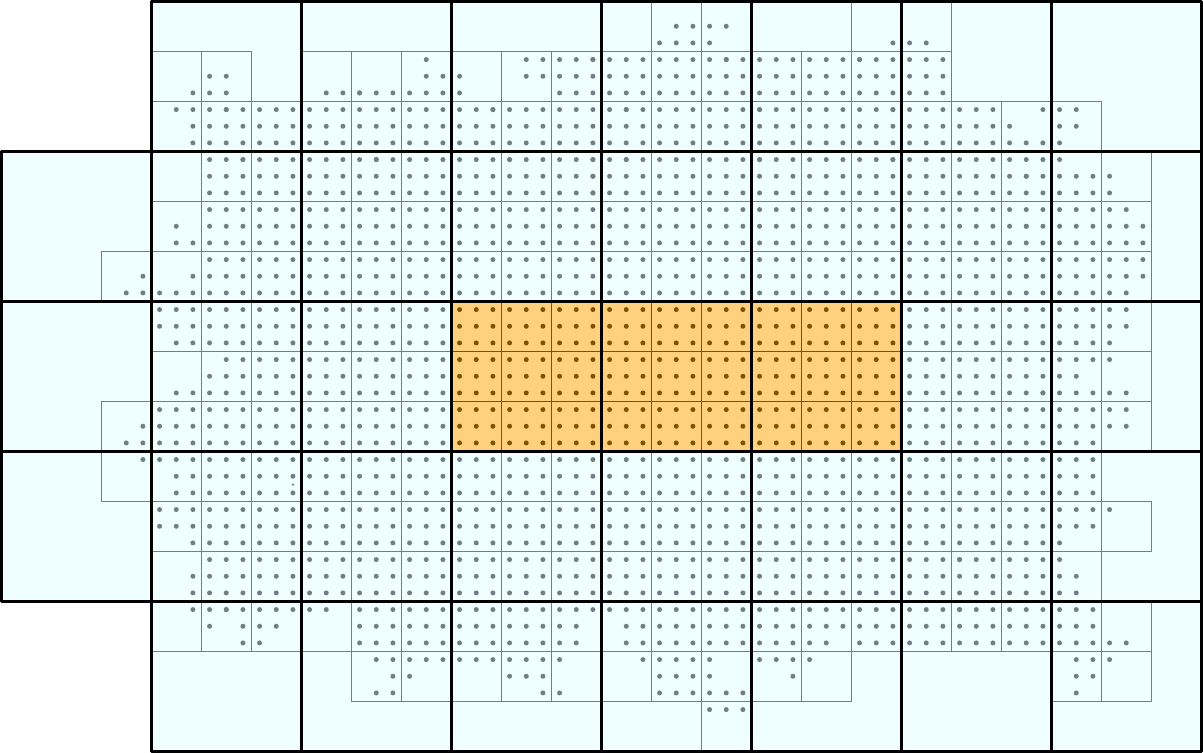}
		\caption{Coarse-graining of a finite volume \(\Lambda\): the small boxes are ``blocked spins'' (\(\sigma\) lives on the disks, \(X=\psi(\sigma)\) lives on the small boxes), the orange large boxes are \([\Lambda]_{\bulk}\), and the blue ones are \([\Lambda]_{\bnd}\).}
		\label{Fig:Lambda_CG}
	\end{figure}
	
	\subsection{Good configurations and canonical paths}
	
	Let \(\Lambda\Subset \Z^2\). Let \(\omega\in \Omega_{\calR(\Lambda)}\). A point \(u\in \partial^{\rmi}\Lambda\) is called a \emph{Markov candidate} (for \(\Lambda\)) if there is \(\zeta\in \Omega_{\calR(\Lambda^c)}\) such that
	\begin{equation*}
		(\omega_{\calR(\Lambda)}\zeta_{\calR(\Lambda^c)})_{\calR(\Lambda_1(x))}\in \MarkovSet_x.
	\end{equation*}Denote \(\mathrm{MC}_{\Lambda}(\omega)\) the set of Markov candidates for \(\Lambda\) in \(\omega\) (\(\mathrm{MC}_{\Lambda}(\omega)\subset \partial^{\rmi}\Lambda\)). Also, for \(\Lambda\Subset \Z^2\) such that \(\Lambda\setminus \partial^{\rmi}\Lambda\) is connected, define
	\begin{equation*}
		\calM_{\Lambda}(\omega) = \{x\in \Lambda\setminus \partial^{\rmi}\Lambda:\ \omega_{\calR(\Lambda_1(x))}\in \MarkovSet_x\},
	\end{equation*}and let \(\calC_{\Lambda}(\omega)\) be the connected component of \(\calM_{\Lambda}(\omega)\) with largest diameter (for the nearest-neighbour connectivity) when it is well defined, and \(\calC_{\Lambda}(\omega)= \varnothing\) else. Finally, set
	\begin{equation*}
		\tilde{\calC}_{\Lambda}(\omega) := \begin{cases}
			\big\{x\in \Lambda:\ x\in \calC_{\Lambda}(\omega) \text{ or } x\in \mathrm{MC}_{\Lambda}(\omega), \rmd_2(x,\calC_{\Lambda}(\omega)) = 1\big\} & \text{ if } \calC_{\Lambda}(\omega)\neq \varnothing\\
			\varnothing & \text{ else}
		\end{cases}.
	\end{equation*}
	
	\subsubsection*{Good site-block configurations}
	
	Say that \(\omega\in \Omega_{\calR(\rmB(i))}\) is a \emph{good configuration} if all of the following hold
	\begin{itemize}
		\item \(\tilde{\calC}_{\rmB(i)}(\omega) \neq \varnothing\),
		\item for every \(j\) such that \(\{i,j\}\in E_{\Gamma}\), the number of \(k\in \{ 1,\dots, 3r^L \}\) such that \(\tilde{\calC}\cap \varrho(i,j,k) \neq \varnothing\) is at least \(0.75\cdot 3r^L\),
		\item for every \(j\) such that \(\{i,j\}\in E_{\Gamma}\), the number of \(k\in \{1,\dots, 3r^{L} \}\) such that \(\tilde{\calC}\cap \varrho(i,j,-k)\neq \varnothing\) is at least \(0.75\cdot 3r^{L}\).
	\end{itemize}Denote \(\good_i\subset \Omega_{\calR(\rmB(i))}\) the set of good configurations for the block \(\rmB(i)\).

	\subsubsection*{Good micro-edge-block configurations and canonical paths}
	
	For \(e=\{i,j\}\in E_{\Gamma}\), \(k\in \{\pm 1,\dots, \pm 2^{L_1}\}\), say that \(\omega\in \Omega_{\calR(E_{e}^k)}\) is a \emph{good configuration} if all of the following hold
	\begin{equation*}
		\tilde{\calC}_{E_{e}^k}(\omega)\neq \varnothing,\quad \tilde{\calC}_{E_{e}^k}(\omega) \cap \varrho(e,i,k)\neq \varnothing,\quad \tilde{\calC}_{E_{e}^k}(\omega) \cap \varrho(e,j,k)\neq \varnothing.
	\end{equation*}Denote \(\good_{e}^k\subset \Omega_{\calR(E_{e}^k)}\) the set of good configurations for the block \(E_{e}^k\).
	
	For \(e=\{i,j\}, k\in \{\pm 1,\dots,\pm r^L\}\), \(\omega\in \good_i\), \(\eta\in \good_{e}^k\), if
	\begin{equation*}
		\tilde{\calC}_{\rmB(i)}(\omega) \cap \varrho(i,j,k)\neq \varnothing,
	\end{equation*}let \(\gamma_{i,e}(\omega,\eta)\subset F_{e}^{k}\setminus E_{e}^k\) be a fixed self avoiding path of length at most \(2L\) connecting \(\tilde{\calC}_{\rmB(i)}(\omega)\) to \(\tilde{\calC}_{E_{e}^k}(\eta)\). These paths will be called \emph{canonical paths} and the sets \(\Lambda_1(\gamma_{i,e}(\omega,\eta))\setminus (\rmB(i)\cup E_{e}^k)\) \emph{canonical sets}

	\subsubsection*{Estimates on good configurations probabilities}
	
	The relevant results concerning good configurations are the following.
	
	\begin{claim}
		\label{claim:good_block_proba}
		Suppose that~\ref{hyp:Markov:Decoupling_circuits}, and~\ref{hyp:Markov:local_Markov_bulk} with \(\pbulk\) close enough to \(1\) hold. Then, there is \(L_0 \geq 1\), such that for any \(L\geq L_0\), and any \(i\in \Gamma\),
		\begin{equation*}
			\nu(X_{\rmB(i)} \in \good_i) \geq 1 - e^{-L_1}
		\end{equation*}
	\end{claim}
	\begin{proof}
		Let \(i\in \Gamma\), \(\Lambda = \rmB(i)\). Let \(X\sim \nu\circ \psi^{-1}\). Let \(Y_x = \mathds{1}_{ \MarkovSet_x}(X_{\Lambda_1(x)})\). As used in the proof of Theorem~\ref{thm:weakMix_to_ratioWeakMix}, the results of~\cite{Liggett+Schonmann+Stacey-1997} and~\ref{hyp:Markov:local_Markov_bulk} imply that \(Y\) can be coupled to a homogeneous Bernoulli site percolation of parameter \(p\) going to \(1\) as \(\pbulk\) goes to \(1\), denoted \(Y'\) (\(Y,Y'\) are defined on the same space, and \(Y'\leq Y\)). When defined, let \(\calC=\calC_{\Lambda}(X)\), and \(\calC'\) be the cluster of \(Y'\) restricted to \(\Lambda\setminus \partial^{\rmi}\Lambda\) having the largest diameter. Also, let \(\tilde{\calC} = \tilde{\calC}_{\Lambda}(X)\), and
		\begin{equation*}
			\tilde{\calC}' = \calC'\cup\{x\in \partial^{\rmi}\Lambda:\ Y_x' = 1,\ \rmd_2(x,\calC) = 1\}.
		\end{equation*}
		
		First note that (by planarity) when there is both a left-to-right and a top-to-bottom crossing of \(\Lambda\setminus \partial^{\rmi}\Lambda\) in \(Y'\) (event denoted \(\mathrm{Cross}\)), \(\calC,\calC'\) is well defined, and \(\calC'\subset \calC\), \(\tilde{\calC}'\subset \tilde{\calC}\). The event of having these crossing is then larger than 
		\begin{equation*}
			P(\mathrm{Cross}) \geq 1-2 (2L_1)^2e^{-2L_1} \geq 1-e^{-L_1}/3
		\end{equation*}for \(p\) large enough (independently of \(L\)), and \(L\) large enough.
		
		Work now under the assumption that the left-to-right and top-to-bottom crossing exist. The event \(\good_i^c\) is then that there is \(j\sim i\) and \(\circ\in \{-,+\}\) such that the number of \(k\)'s in \(\{\circ 1,\dots,\circ r^{L}\}\) with \(|\tilde{\calC} \cap \varrho(i,j,k)|\neq \varnothing\) is less than \(0.75\cdot r^{L}\). This second condition implies
		\begin{equation}
			\label{eq:prf:goodBlock:eq1}
			|\tilde{\calC} \cap \partial^{\rmi}\Lambda| \leq |\partial^{\rmi}\Lambda|- 0.25 Lr^{L},
		\end{equation}as \(|\varrho(i,j,k)| = L\). Now,
		\begin{equation*}
			|\partial^{\rmi}\Lambda| = 8L_1 = 24Lr^L,
		\end{equation*}so that~\eqref{eq:prf:goodBlock:eq1} is
		\begin{equation}
			\label{eq:prf:goodBlock:eq2}
			|\tilde{\calC} \cap \partial^{\rmi}\Lambda| \leq |\partial^{\rmi}\Lambda| - \frac{1}{96}|\partial^{\rmi}\Lambda| = \frac{95}{96}|\partial^{\rmi}\Lambda|.
		\end{equation}
		As \(\tilde{\calC}\supset \tilde{\calC}'\), the above event is included in \(|\tilde{\calC}' \cap \partial^{\rmi}\Lambda| \leq \frac{95}{96}|\partial^{\rmi}\Lambda|\). Now, for any \(n\geq 1\),
		\begin{equation*}
			P\big(|\tilde{\calC}'\cap \partial^{\rmi}\Lambda| = n\big) = \sum_{m\geq n} p^{n}(1-p)^{m-n}\binom{m}{n}P\big(|\calC'\cap \partial^{\rmi}\Lambda_{L_1-1}(i)| = m\big),
		\end{equation*}as there is exactly one site in \(\partial^{\rmi}\Lambda\) at distance \(1\) from \(\calC'\) per site in \(\calC'\cap \partial^{\rmi}\Lambda_{L_1-1}(i)\). Letting \(M= \lfloor\frac{191}{192}|\partial^{\rmi}\Lambda| \rfloor\), and \(\mathrm{Bin}\) denote a binomial random variable of parameters \(M,p\), one gets that
		\begin{equation*}
			P\Big(|\tilde{\calC}'\cap \partial^{\rmi}\Lambda| \leq \frac{95}{96}|\partial^{\rmi}\Lambda|\Big) \leq P\big(|\calC'\cap \partial^{\rmi}\Lambda_{L_1-1}(i)| \leq M\big) + P\Big(\mathrm{Bin}\leq \frac{95}{96}|\partial^{\rmi}\Lambda|\Big).
		\end{equation*}
		
		One can then apply Lemma~\ref{lem:app:good_box} to obtain that for \(p\) close enough to \(1\) (independently of \(L\)) and \(L\) large enough,
		\begin{equation*}
			P\big(|\calC'\cap \partial^{\rmi}\Lambda_{L_1-1}(i)| \leq M\big) \leq e^{-L_1}/24.
		\end{equation*}Moreover, by standard large deviation estimates for binomial random variables, for \(p\) close enough to \(1\) (independent of \(L\)),
		\begin{equation*}
			P\Big(\mathrm{Bin}\leq \frac{95}{96}|\partial^{\rmi}\Lambda|\Big) \leq e^{-L_1}/24.
		\end{equation*}This implies that (by a union bound over the choices for \(j\sim i\) and \(\circ\in \{-,+\}\))
		\begin{equation*}
			P(X\in \good_{i}^c, \mathrm{Cross})\leq 8\cdot 2 e^{-L_1}/24 = 2 e^{-L_1}/3.
		\end{equation*}Putting everything together gives
		\begin{equation*}
			P(X\in \good_{i}^c) \leq 1-P(\mathrm{Cross}) + P(X\in \good_{i}^c, \mathrm{Cross}) \leq e^{-L_1},
		\end{equation*}which is the claim.
	\end{proof}
	
	\begin{claim}
		\label{claim:good_edge_proba}
		Suppose that~\ref{hyp:Markov:Decoupling_circuits}, and~\ref{hyp:Markov:local_Markov_bulk} with \(\pbulk\) close enough to \(1\) hold. Then, for any \(L\geq 1\), and any \(e\in E_{\Gamma}\), \(k\in \{\pm 1,\dots \pm r^L\}\),
		\begin{equation*}
			\nu(X_{E_e^k} \in \good_e^k) \geq 1 - (aL)^2e^{-2L}.
		\end{equation*}
	\end{claim}
	\begin{proof}
		Proceed as in the proof of Claim~\ref{claim:good_block_proba}: compare \(\tilde{\calC}_{E_{e}^k}\) with a Bernoulli percolation cluster (of parameter \(p\) close to one when \(\pbulk\) is close to \(1\)). The absence of long crossing is then equivalent to the presence of short \(*\)-crossing of closed sites, which has exponential decay for \(p\) large enough (\(p\) larger than \(1-e^{-2}/16\) is amply sufficient).
	\end{proof}

	\subsection{Construction of the sampling functions}
	\label{sec:sampling_fct}
	
	\subsubsection*{Bulk spin functions}
	
	There are
	two
	kinds of ``blocks'' in the coarse-graining of Section~\ref{sec:prf_main:CG}
	\begin{itemize}
		\item blocks associated to elements of \(\Gamma\): \(\rmB(i), i\in\Gamma\),
		\item blocks associated to elements of \(E_{\Gamma}\times \{\pm 1,\dots, \pm r^{L}\}\): \(E_{e}^k\), \(e\in E_{\Gamma}\), \(k\in \{\pm 1,\dots, \pm r^{L}\}\).
	\end{itemize}
	
	Introduce the thresholds\footnote{They correspond to the amounts of mass that are common to all volumes and conditioning outside of some suitable neighbourhood of the concerned blocs.} associated to (in order): \(\Gamma\)-indexed blocks, and labelled-\(E_{\Gamma}\)-blocks:
	\begin{gather*}
		s_i^*(\omega) = \inf_{t}\inf_{\Lambda\supset \calR(\rmB'(i))}\inf_{\xi} \nu_{\Lambda,\calR(\rmB(i))}^{t,\xi}\big( X_{\rmB(i)} = \omega \big),
		\quad \omega\in \Omega_{\calR(\rmB(i))};
		\\
		s_{e,k}^*(\omega) = \inf_{t}\inf_{\Lambda\supset \calR(\Lambda_{L}(E_e^k))} \inf_{\xi} \nu_{\Lambda,\calR(\Lambda_L(E^{k}_e))}^{t,\xi}\big( X_{E^k_{e}} = \omega\big),
		\quad \omega\in \Omega_{\calR(E^k_{e})};
	\end{gather*}where the inf over \(\xi\) is always over configurations having positive probability under \(\nu_{\Lambda}^t\). Also introduce the corresponding renormalized quantities:
	\begin{gather*}
		S_{a}^* =\sum_{\eta\in \Omega_{D_a}} s_{a}^*(\eta),\quad  p_{a}^*(\omega) = \frac{1}{S_{a}^*}s_{a}^*(\omega),
	\end{gather*}where \(a \in \{i,(e,k)\}\) and \(D_a\) is the corresponding volume. Denote \(P_a^*\) the probability measure with density \(p_a^*\).
	
	Define now for each \(i\in [\Lambda]_{\bulk}\), each \( \Delta\subset \Lambda\setminus \calR(\rmB'(i))\), and each \(\zeta \in \Omega_{\Delta}\)
	\begin{equation*}
		r(\omega,i,\Delta,\zeta) = \nu_{\Lambda,\Lambda\setminus \Delta}^{t,\zeta}\big( X_{\rmB(i)} = \omega \big) - s_i^*(\omega) \geq 0.
	\end{equation*}In the same fashion, for \(e=\{i,j\}\in E_{\Gamma}\) with \(i,j\in [\Lambda]_{\bulk}\), \(k \in \{\pm 1,\dots, \pm r^{L}\}\), \( \Delta\subset \Lambda\setminus \calR(\Lambda_{L}(E_{e}^k))\), and \(\zeta \in \Omega_{\Delta}\), let
	\begin{equation*}
		r\big(\omega,(e,k),\Delta,\zeta\big) = \nu_{\Lambda,\Lambda\setminus \Delta}^{t,\zeta}\big( X_{E_{e}^k} = \omega \big) - s_{e,k}^*(\omega) \geq 0.
	\end{equation*}
	
	To avoid repetition, let \(D \in \{ \rmB(i),\ E_{e}^k\}\) be fixed, and let \(a\in \{i,(e,k)\}\) be the corresponding block label, and \(D'\in \{\rmB'(i), \Lambda_{L}(E_{e}^k)\}\) be the neighbourhood of \(D\). The associated function is then defined as follows. Let \(\{\omega_1,\omega_2,\dots\}\) be an enumeration of the elements of \(\Omega_{\calR(D)}\) (which is at most countable). For \(\Delta \subset \Lambda\setminus \calR(D')\), and \(\zeta\in \Omega_{\Delta}\), let
	\begin{gather*}
		\alpha_{0} = 0,\quad \alpha_{m} = \sum_{l=1}^m s_{a}^*(\omega_l),\quad \beta_{0} = 0,\quad \beta_{m}(a,\Delta,\zeta) = \sum_{l=1}^m r(\omega_l,a,\Delta,\zeta), \\
		I_{m} = [\alpha_{m-1},\alpha_{m}),\quad J_{m}(a,\Delta,\zeta) = \big[S_{a}^* + \beta_{m-1}(a,\Delta,\zeta), S_{a}^* + \beta_{m}(a,\Delta,\zeta)\big).
	\end{gather*}Then, define the functions \(f_{a,\Delta}^{\zeta}:[0,1)\to \Omega_{\calR(D)}\) via
	\begin{equation*}
		f_{a,\Delta}^{\zeta}(u) = \omega_m \text{ if } u\in I_{m} \sqcup J_{m}(a,\Delta,\zeta).
	\end{equation*}
	\begin{remark}
		If \(U\) is uniform on \([0,1)\), \(f_{a,\Delta}^{\zeta}(U)\) has law \(\nu_{\Lambda,\Lambda\setminus \Delta}^{t,\zeta}\) restricted to \(\calR(D)\). Moreover, if \(0\leq u< S_a^*\), \(f_{a,\Delta}^{\zeta}(u) = f_{a,\tilde{\Delta}}^{\tilde{\zeta}}(u)\) for any \(\Delta,\tilde{\Delta}\subset \Lambda\setminus \calR(D')\) and \(\zeta\in \Omega_{\Delta}\), \(\tilde{\zeta}\in \Omega_{\tilde{\Delta}}\).
	\end{remark}
	
	\subsubsection*{Boundary spin functions}
	
	The functions \(f\)'s will be used in the bulk and their structure plays with the mixing properties of the model. For the remaining sites, the only thing used will be finite energy.
	
	For \(\Delta\subset \Lambda\), \(W\subset \calS(\Lambda)\setminus \calS(\Delta)\), and \(\zeta\in \Omega_{\Delta}\), let \(g_{W,\Delta}^{\zeta}:[0,1)\to \Omega_{\calR(W)\cap \Lambda}\) be a measurable function such that
	\begin{itemize}
		\item if \(U\) is a uniform random variable on \([0,1)\), \(g_{W,\Delta}^{\zeta}(U)\sim \nu_{\Lambda,\Lambda\setminus \Delta}^{t,\zeta}|_{\calR(W)\cap \Lambda}\),
		\item if \(u< \theta^{|W|}\), then \(\big(g_{W,\Delta}^{\zeta}(u)\big)_{\Lambda_l(x)} = \MarkovEl_{x,\Lambda}^t\) for any \(x\in W\) such that \(\Lambda_{l+l'}(x)\cap \Lambda^c\neq \varnothing\), and \(\big(g_{W,\Delta}^{\zeta}(u)\big)_{\Lambda_l(x)} = \MarkovEl_{x}\) for \(x\in W\) such that \(\Lambda_{l+l'}(x) \subset \Lambda\).
	\end{itemize}Such functions can be constructed in the same fashion as the \(f\)'s. The fact that the second point can be realized follows straightforwardly from Hypotheses~\ref{hyp:Markov:local_Markov_finite_energy}.

	\subsubsection*{Sampling estimates}
	
	The relevant claims concerning the thresholds/sampling functions are the following.
	\begin{claim}
		\label{claim:threshold_good_blocks}
		Suppose hypotheses~\ref{hyp:Mix:inf_vol_meas}, and~\ref{hyp:Mix:exp_rel_density} holds. Then, for any \( a\geq \frac{4\ln(r)}{\cmix l'}\), any \(i\in \Gamma\), and any \(A\subset \Omega_{\calR(\rmB(i))}\),
		\begin{equation*}
			S_i^* P_i^*(A) = \sum_{\omega\in \Omega_{\calR(\rmB(i))}} \mathds{1}_{A}(\omega) s_i^*(\omega) \geq \nu(X_{\rmB(i)}\in A) - 81\Cmix (Ll')^2 r^{-2L}.
		\end{equation*}
	\end{claim}
	\begin{proof}
		By~\ref{hyp:Mix:inf_vol_meas}, one has that for any \(\omega\in \Omega_{\calR(\rmB(i))}\), \(\nu(X_{\rmB(i)} = \omega) \geq s_i^*(\omega)\). In particular,
		\begin{equation*}
			\sum_{\omega\in \calR(\Omega_{\rmB(i))}} \mathds{1}_{A^c}(\omega) s_i^*(\omega) \leq \sum_{\omega\in \Omega_{\calR(\rmB(i))}} \mathds{1}_{A^c}(\omega) \nu(X_{\rmB(i)} = \omega) = \nu(X_{\rmB(i)}\in A^c).
		\end{equation*}
		Then, by~\ref{hyp:Mix:exp_rel_density}, the definitions of \(L_1\), and the hypotheses on \(a\),
		\begin{multline*}
			\sum_{\omega\in \Omega_{\calR(\rmB(i))}} s_i^*(\omega) 
			\geq 1- 9\Cmix (L_1l')^2e^{-\cmix aLl'}
			\\
			= 1- 81\Cmix (Ll')^2 e^{2\ln(r)L-\cmix aLl'}
			\geq 1- 81\Cmix (Ll')^2 e^{-2\ln(r)L},
		\end{multline*}where \(\Cmix,\cmix\) are the values given in~\ref{hyp:Mix:exp_rel_density}, and \(C\) is some universal constant. Combining the two, one obtains
		\begin{multline*}
			\sum_{\omega\in \Omega_{\calR(\rmB(i))}} \mathds{1}_{A}(\omega) s_i^*(\omega)
			= \sum_{\omega\in \Omega_{\calR(\rmB(i))}} s_i^*(\omega) - \sum_{\omega\in \Omega_{\calR(\rmB(i))}} \mathds{1}_{A^c}(\omega) s_i^*(\omega)
			\\
			\geq 1- C\Cmix (Ll')^2 e^{-2\ln(r)L} - \nu(X_{\rmB(i)}\in A^c)
			= \nu(X_{\rmB(i)}\in A) - 81\Cmix (Ll')^2 r^{-2L}.
		\end{multline*}
	\end{proof}
	
	\begin{claim}
		\label{claim:threshold_good_edges}
		Suppose hypotheses~\ref{hyp:Mix:inf_vol_meas}, and~\ref{hyp:Mix:exp_rel_density} holds. Then, for any \(e\in E_{\Gamma}\), \(k\in \{\pm 1,\dots,\pm r^{L}\}\), and any \(A\subset \Omega_{E_e^k}\),
		\begin{equation*}
			S_{e,k}^* P_{e,k}^*(A) = \sum_{\omega\in \Omega_{E_{e}^k}} \mathds{1}_{A}(\omega) s_{e,k}^*(\omega) \geq \nu(X_{E_e^k}\in A\big) - 2a\Cmix (Ll')^2 e^{-\cmix Ll'}.
		\end{equation*}
	\end{claim}
	\begin{proof}
		Proceeding as in the proof of Claim~\ref{claim:threshold_good_blocks}, one obtains
		\begin{gather*}
			\sum_{\omega\in \Omega_{\calR(E_e^k)}} \mathds{1}_{A^c}(\omega) s_{e,k}^*(\omega) \leq \nu(X_{E_e^k}\in A^c\big),\\
			\sum_{\omega\in \Omega_{\calR(E_e^k)}} s_{e,k}^*(\omega) \geq 1 - 2a\Cmix (Ll')^2 e^{-\cmix Ll'}.
		\end{gather*}The claim follows as in the proof of Claim~\ref{claim:threshold_good_blocks}.
	\end{proof}

	\subsection{Definition of the patch exploration process}
	\label{sec:exploration_process}
	
	To sample the spin configuration, the functions used are: for \(\Delta\subset \Lambda,\ W\subset \calS(\Lambda)\setminus \calS(\Delta),\ \xi\in \Omega_{\Delta}\),
	\begin{equation*}
		F_{\Delta;W}^{\xi} =\begin{cases}
			f_{i,\Delta}^{\xi} & \text{ if } W = \rmB(i),\ i\in [\Lambda]_{\bulk},\\
			f_{(e,k),\Delta}^{\xi} & \text{ if } W = E_{e,k},\ e\subset [\Lambda]_{\bulk},\\
			g_{W,\Delta}^{\xi} & \text{ else}.
		\end{cases}.
	\end{equation*}
	
	Consider the planar labelled (multi)graph \(G_{\ell}= (\Gamma, E_{\ell})\) (\(\ell\) for \emph{labelled}) with vertex set \(\Gamma\) and \(2 r^L\) edges between neighbouring sites, labelled \(\pm 1,\dots \pm r^L\), from the ``inside'' to the ``outside'' (see Figure~\ref{Fig:multigraph})
	
	\begin{figure}[h]
		\centering
		\includegraphics[scale=0.8]{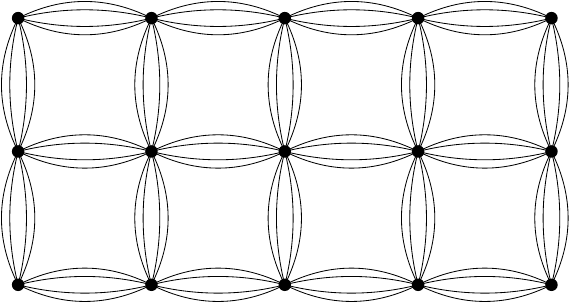}
		\hspace*{1cm}
		\includegraphics[scale=0.8]{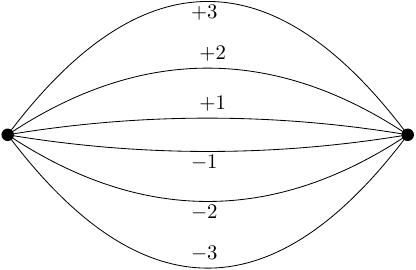}
		\caption{The labelled multigraph \(G\).}
		\label{Fig:multigraph}
	\end{figure}
	
	Now, the sequence of set valued functions \(\calA_{k}, k\geq 0\) will be given by the simultaneous exploration of a percolation configuration and of a spin configuration. It is given by Algorithm~\ref{alg:patch_exploration}, which requires a few preparations. Consider site-edges percolation configurations on \(G_{\ell}\) with all sites of \([\Lambda]_{\ext}\) open, all edges with at least one endpoint in \([\Lambda]_{\ext}\) open, and all edges with both endpoints in \(\Gamma\setminus [\Lambda]_{\bulk}\) open:
	\begin{multline*}
		\Theta_{G_{\ell},\Lambda} = \big\{\eta\in \{0,1\}^{\Gamma\cup E_{\ell}}:\\
		\eta_i = 1 \text{ if } i\in [\Lambda]_{\ext},\ \eta_{(e,k)} = 1 \text{ if } e\cap [\Lambda]_{\ext}\neq \varnothing \text{ or } e\subset \Gamma\setminus [\Lambda]_{\bulk} \big\}
	\end{multline*}
	
	For \(A\subset \Gamma\), \(\eta\in \Theta_{G_{\ell},\Lambda}\), let \(\gamma_{A}(\eta)\) be the innermost set of open sites and edges surrounding \(A\) in \(\eta\) (it is a self-avoiding path if \(A\) is connected, but can also be a collection of disjoints self-avoiding paths). For \(k\geq 0\), consider functions
	\begin{gather*}
		\rmW_k: \{(B,\eta_B):\ \eta_B\in\{0,1\}^{B} ,\ B\subset \Gamma\cup E_{\ell}, |B| = k \}\to \Gamma\cup E_{\ell} ,\\
		\tau_k: \{(B,\eta_B):\ \eta_B\in\{0,1\}^{B} ,\ B\subset \Gamma\cup E_{\ell}, |B| = k \} \to \{0,1\},
	\end{gather*}which satisfy
	\begin{itemize}
		\item \(\rmW_k(B,\eta_B)\notin B\) for any \(B,\eta_B\),
		\item \(\tau_k(B,\eta_B) = 1\) implies \(\tau_{k+1}(B',\eta_{B'}) = 1\) for any \(B\subset B'\) and \(\eta_B,\eta_{B'}\) satisfying \(\eta_B(x) = \eta_{B'}(x)\) for all \(x\in B\).
	\end{itemize}For a percolation configuration \(\eta\in \Theta_{G_{\ell},\Lambda}\), define then inductively (\(\rmW_0\) is a constant function)
	\begin{gather*}
		\calW_0 = \rmW_0,\quad \calB_1 = \{\calW_0\},\\
		\calW_{k}(\eta) = \rmW_{k}(\calB_{k}(\eta), \eta_{\calB_{k}}),\quad \calB_{k}(\eta) = \calB_{k-1}(\eta) \cup \{\calW_{k-1}(\eta)\},\\
		\calT(\eta) = \min\big\{k\geq 0:\ \tau_k(\calB_{k}(\eta), \eta_{\calB_{k}})= 1\big\}.
	\end{gather*}One then has that \((\calB_{\calT}, \calT)= (B,t)\) is an event measurable with respect to the restriction of \(\eta\) to \(B\) (i.e., \(\eta_B\)). The process \(\big(\calB_k(\eta), \eta_{\calB_k}\big)_{t=0}^{\calT(\eta)}\) is called a \emph{stopped exploration of the percolation configuration \(\eta\)}. Consider such a stopped exploration so that
	\begin{itemize}
		\item \(\calB_{\calT}(\eta)\) is \(\gamma_{[\Delta_2]}(\eta)\) union some edges and sites surrounded by \(\gamma_{[\Delta_2]}(\eta)\) (an exploration of \(\gamma_{[\Delta_2]}\) from inside);
		\item \(x\in \calB_{k}(\eta)\setminus \calB_{k-1}(\eta)\) can be an element \(\{i,j\}\) of \(E_{\ell}\) only if \(i,j\in \calB_{k-1}(\eta)\) (edges are explored only when their endpoints are already explored). In particular, \(\calW_0 \in \Gamma\);
		\item the edge \((e,k)\) with \(e=\{i,j\}\) is only explored if \(\eta(i) = \eta(j) = 1\);
		\item no site in \([\Delta_2]\) is explored, and no edge with an endpoint in \([\Delta_2]\) is ever explored.
	\end{itemize}Existence of such exploration process is by now a standard thing in percolation theory.

	Note that for the purpose of proving Theorem~\ref{thm:main}, it is sufficient to assume that \(\Delta_2 = \calR(\calS(\Delta_2))\cap\Lambda\) as the percolation model does not see the difference for \(\ell\) large enough, and the measures on the L.H.S. are mixtures of the case \(\Delta_2 = \calR(\calS(\Delta_2))\cap\Lambda\). Now, for \(\xi\in \Omega_{\Delta_2}\) and a sequence \(u_0,u_1,\dots\in[0,1)\), define \((\calA,\vartheta_{\calA}, \calB, \eta_{\calB})\) with \(\calA\subset \calS(\Lambda)\), \(\vartheta_{\calA}\in \Omega_{\calR(\calA)}\), \(\calB\Subset \Gamma\cup E_{\ell}\), \(\eta\in \{0,1\}^{\calB}\) via Algorithm~\ref{alg:patch_exploration}. \((\calA,\vartheta_{\calA})\) is a patch exploration of \(\nu_{\Lambda,\Lambda\setminus \Delta_2}^{t,\xi}\circ \psi^{-1}\) stopped at a stopping time. \((\eta_{\calB}, \calB)\) is a stopped exploration of some percolation model. The properties which follow from the construction are:
	\begin{itemize}
		\item \(\calB\) is the union of a collection of closed self-avoiding site-edge paths, \(\gamma\), which surrounds \([\Delta_2]\) with some edges and sites in the set \([\Delta]_2\)-surrounded by \(\gamma\). \(\eta_{\calB}\) is open on \(\gamma\);
		\item \((\calB, \calA, \eta_{\calB})\) does not depend on \(\xi\);
		\item if \(i\in \Gamma\cap \calB\) is such that \(\eta(i) = 1\), then \(\vartheta_{\rmB(i)}\) is good and does not depend on \(\xi\);
		\item if \((e,k)\in E_{\ell}\cap \calB\) is such that \(\eta(e,k) = 1\), then \(\vartheta_{\calA \cap \Lambda_1(F_{e,k})}\) does not depend on \(\xi\);
		\item denoting
		\begin{equation*}
			\calD = \calR\Big(\bigcup_{i\in \Gamma\cap \gamma}\rmB(i)\cup \bigcup_{(e,k)\in E_{\ell}\cap \gamma} \Lambda_1(F_{e,k})\Big)\cap \Lambda,
		\end{equation*}one has that \(\psi^{-1}(\vartheta)\) restricted to \(\calD\) is a decoupling configuration in the sense of Hypotheses~\ref{hyp:Markov:Decoupling_circuits}.
	\end{itemize}
	
	\begin{algorithm}[H]
		\label{alg:patch_exploration}
		\KwData{\(u=(u_0,u_1,\dots)\), \(\xi\in \Omega_{\Delta_2}\)}
		\KwResult{\(\calS(\Delta_2)\subset \calA\subset \Lambda\), \(\vartheta_{\calA} \in \Omega_{\calA}\), \(\calB \Subset \Gamma\cup E_{\ell}\), \(\eta_{\calB}\in \{0,1\}^{\calB}\)}
		Initialise \(l=0\), \(m=0\), \(\calA_0 = \calS(\Delta_2)\), \(\vartheta_{\calS(\Delta_2)} = \psi_{\Lambda}(\xi)\)\;
		\While{\(\tau_l(\calB_l,\eta_{\rmB_l}) = 0\)}{
			Set \(w = \rmW_l(\calB_l, \eta_{\calB_l})\)\;
			\If{\(w\in \Gamma\)}{
				Update \(m=m+1\)\;
				Set \(\vartheta_{\rmB(w)} = F_{\calA_{m-1};\rmB(w)}^{\vartheta_{\calA_{m-1}}} (u_m)\), \(\calA_{m} = \calA_{m-1}\cup \rmB(w)\)\;
				\eIf{\(u_m<S_w^*\) \Aand \(\theta_{\rmB(w)}\in \good_w \)}{
					Set \(\eta(w) = 1\)\;
				}{
					Set \(\eta(w) = 0\)\;
				}
			}
			\If{\(w = (e,k) = (\{i,j\},k) \in E_{\ell}\)}{
				Set \(\vartheta_{E_{e}^k} = F_{\calA_{m-1};E_{e}^k}^{\vartheta_{\calA_{m-1}}} (u_{m+1})\), \(\calA_{m+1} = \calA_{m}\cup E_{e}^k\)\;
				\eIf{\(u_{m+1}<S_{(e,k)}^*\) \Aand \(\vartheta_{E_{e}^k}\in \good_{e}^k \)}{
					Let \(A_i= \Lambda_1(\gamma_{e,i}(\vartheta_{\rmB(i)},\vartheta_{ E_{e}^k}))\), \(A_j= \Lambda_1(\gamma_{e,j}(\vartheta_{\rmB(j)},\vartheta_{ E_{e}^k}))\) be the canonical sets associated respectively to the good configurations \(\vartheta_{\rmB(i)}, \vartheta_{E_{e}^k}\) and \(\vartheta_{\rmB(j)}, \vartheta_{E_{e}^k}\)\;
					Set \(\calA_{m+2} = \calA_{m+1} \cup A_i\), \(\calA_{m+3} = \calA_{m+2}\cup A_j\)\;
					Set \(\vartheta_{A_i} = F_{\calA_{m+1};A_i}^{\vartheta_{\calA_{m+1}}}(u_{m+2}) \), and \(\vartheta_{A_j} = F_{\calA_{m+2};A_j}^{\vartheta_{\calA_{m+2}}}(u_{m+3}) \)\;
					\If{\(u_{m+2}< \theta^{2L}\) \Aand \(u_{m+3}< \theta^{2L}\)}{
						Set \(\eta(w) = 1\)\;
					}
				}{
					Set \(\eta(w) = 0\)\;
					Set \(\vartheta_{(F_{e}^k\setminus E_{e}^k)\cap \Lambda_K(i)} = F_{\calA_{m+1};(F_{e}^k\setminus E_{e}^k)\cap \Lambda_K(i)}^{\vartheta_{\calA_{m+1}}} (u_{m+2})\), \(\calA_{m+2} = \calA_{m+1}\cup ((F_{e}^k\setminus E_{e}^k)\cap \Lambda_K(i)) \)\;
					Set \(\vartheta_{(F_{e}^k\setminus E_{e}^k)\cap \Lambda_K(j)} = F_{\calA_{m+2};(F_{e}^k\setminus E_{e}^k)\cap \Lambda_K(j)}^{\vartheta_{\calA_{m+2}}} (u_{m+3})\), \(\calA_{m+3} = \calA_{m+2}\cup ((F_{e}^k\setminus E_{e}^k)\cap \Lambda_K(j)) \)\;
				}
				Update \(m=m+3\)\;
			}
			Update \(l=l+1\)\;
			Set \(\calB_{l} = \calB_{l-1}\cup \{w\}\)\;
		}
		\Return{\(\calB= \calB_l, \calA = \calA_{m}, \vartheta_{\calA}, \eta_{\calB}\)\;}
		\caption{Patch sampling algorithm. The spins are sampled in one go when a site of the percolation configuration is explored, while the spins are sampled first in \(E_{e}^k\) and then in \(F_{e}^k\setminus E_{e}^k\) (in two steps) when the edge \((e,k)\) is explored. Note that the first case of edge exploration is well defined as edges are only explored when their endpoints are open (so that the spin configurations on the blocks associated to the endpoints are good).}
	\end{algorithm}
	
	As a consequence of the properties of Algorithm~\ref{alg:patch_exploration}, letting (\(\calD\) defined as above)
	\begin{itemize}
		\item \(\bar{U} = U_0,U_1,\dots\) be an i.i.d. sequence of uniform on \([0,1)\),
		\item \(\calA = \calA(\bar{U}),\quad Y_{\calA}^{\xi} = \psi^{-1}(\vartheta_{\calA}(\bar{U},\xi))\) and\begin{equation*}
			\calJ = \{x\in \Lambda:\ x\in \calD(\bar{U}) \text{ or } x \ \Delta_2\text{-surrounded by } \tilde{\calD}(\bar{U})\},
		\end{equation*}where \(\tilde{\calD} = \calD\cup E_{\calD}\), with \(E_{\calD}\) the set of n.n. edges with both endpoints in \(\calD\),
	\end{itemize}one has that \(\nu_{\Lambda, \Lambda \setminus \calR(\calA)}^{t,Y_{\calA}^{\xi}}|_{\Lambda\setminus\calJ}\) does not depend on \(\xi\). For \(A,B\subset \Lambda\) disjoint, and \(\zeta_A\in \Omega_{A}\), let \(Z_{A}^{\zeta_A} \sim \nu_{\Lambda,\Lambda\setminus A}^{t,\zeta_A}\) be random variables independent of \(\bar{U}\) (all living on the same probability space). Then, by Lemma~\ref{lem:patchExploration_sampling} point 2., and the observation above, one has that \(Y_{\calA}^{\xi} Z_{\calR(\calA)}^{Y_{\calA}^{\xi}} \sim \nu_{\Lambda,\Lambda\setminus \Delta_2}^{t,\xi}\). In particular, for any \(\Delta_1\subset \Lambda\setminus \Delta_2\), any \(D\subset \Omega_{\Delta_1}\), and any \(\xi,\xi'\in \Omega_{\Delta_2}\), (abbreviating \(Y_{\calA} = Y^{\xi}_{\calA}\) and \(Y_{\calA}' = Y^{\xi'}_{\calA}\))
	
	\begin{multline*}
		|\nu_{\Lambda,\Lambda\setminus \Delta_2}^{t,\xi}(\sigma_{\Delta_1}\in D) - \nu_{\Lambda,\Lambda\setminus \Delta_2}^{t,\xi'}(\sigma_{\Delta_1}\in D)|
		=
		\big|\rmE\big(\mathds{1}_{Y_{\calA} Z_{\calA}^{Y_{\calA}}|_{\Delta_1}\in D} - \mathds{1}_{Y_{\calA}' Z_{\calA}^{Y_{\calA}'}|_{\Delta_1}\in D} \big)\big|
		\\
		\leq
		\rmE(\mathds{1}_{\Delta_1\cap \calJ \neq \varnothing}) + \sum_{A}\sum_{J: J\cap \Delta_1= \varnothing} \big|\rmE\big(\mathds{1}_{\calA= A}\mathds{1}_{\calJ= J}(\mathds{1}_{Z_{\calA}^{Y_{\calA}}|_{\Delta_1}\in D} - \mathds{1}_{Z_{\calA}^{Y_{\calA}'}|_{\Delta_1}\in D})\big)\big|
		\\
		=
		\rmE(\mathds{1}_{\Delta_1\cap \calJ \neq \varnothing}) + \sum_{A}\sum_{J: J\cap \Delta_1= \varnothing}
		\Big|\rmE\Big(\mathds{1}_{\calA= A, \calJ = J}\big(\nu_{\Lambda,\Lambda\setminus A}^{t,Y_A}|_{\Delta_1}(D) - \nu_{\Lambda,\Lambda\setminus A}^{t,Y_A'}|_{\Delta_1}(D)\big)\Big)\Big|
		\\
		=
		\rmP(\Delta_1\cap \calJ \neq \varnothing)
		\leq
		\rmP([\Delta_1] \text{ is not separated from } [\Delta_2] \text{ in } \eta),
	\end{multline*}as \(\calA,\calB\) (and thus \(\calD,\calJ\)) do not depend on \(\xi,\xi'\), and, for \(\calJ=J\) such that \(J\cap \Delta_1=\varnothing\), the restriction of \(Y_{\calA}^{\xi} Z_{\calA}^{Y_{\calA}}\) to \(\Delta_1\) is simply the restriction of \(Z_{\calA}^{Y_{\calA}}\) to \(\Delta_1\), and the \(Z\)'s are independent from the \(U\)'s. In the last line, the independence of \(\nu_{\Lambda,\Lambda\setminus A}^{t,Y_A}|_{\Delta_1}\), \(\nu_{\Lambda,\Lambda\setminus A}^{t,Y_A'}|_{\Delta_1}\) on \(\xi,\xi'\) is used (so that the two measures are the same).
	
	To conclude this section, notice that the exploration of \(\eta\) has the property that uniformly over what was explored before, the probability that \(\eta(w) = 1\), with \(w\) the next site to be explored, is at least
	\begin{itemize}
		\item \(\rmp_{\rms} := \inf_{i\in \Gamma} P_i^*(\good_i)S_i^*\) if \(w\in [\Lambda]_{\bulk}\),
		\item \(\rmp_{\rme} := \theta^{4L}\inf_{e\in E_{\Gamma}}\inf_{k\in \{\pm 1,\dots,\pm r^L\}} P_{e,k}^*(\good_e^k)S_{e,k}^*\) if \(w = (e,k)\in E_{\gamma}\times\{\pm 1,\dots, \pm r^L\}\), with \(e\subset [\Lambda]_{\bulk}\),
		\item \(\rmp_{\rmb\rms} := \theta^{(2L_1+1)^2}\) if \(w\in [\Lambda]_{\bnd}\),
		\item \(\rmp_{\rmb\rme} := \theta^{(2a+2)L^2}\) if \(w=(\{i,j\},k)\) with \(i\in [\Lambda]\), \(j\notin[\Lambda]_{\bulk}\),
		\item \(1\) else.
	\end{itemize}
	
	In particular, \(\eta\) stochastically dominates a site-edge Bernoulli percolation on \((\Gamma,E_{\ell})\) with parameters given above. The bounds derived in the previous sections allow to estimate these parameters.
	\begin{claim}
		\label{claim:perco_labelled_bnds}
		There is \(p_0\in(0,1)\) such that, if Hypotheses~\ref{hyp:Mix:exp_rel_density} and~\ref{hyp:Markov:local_Markov_bulk} with \(\pbulk\geq p_0\) hold, one has that for any \(r\geq 2\), any \(a\geq \frac{4\ln(r)}{\cmix l'}\), there is \(L_0 \geq 0\) (depending on \(l\)) such that for any \(L\geq L_0\),
		\begin{equation*}
			\rmp_{\rms} \geq 1-r^{-L},\qquad \rmp_{\rme} \geq \frac{1}{2}\theta^{4L}.
		\end{equation*}
	\end{claim}
	\begin{proof}
		Using Claims~\ref{claim:good_block_proba}, and~\ref{claim:threshold_good_blocks}, one obtains that for any \(i\in \Gamma\),
		\begin{equation*}
			P_i^*(\good_i)S_i^* \geq 1-e^{-L_1} - 81\Cmix (Ll')^2r^{-2L}.
		\end{equation*}
		Using Claims~\ref{claim:good_edge_proba}, and~\ref{claim:threshold_good_edges}, one obtains that for any \(e\in E_{\Gamma}, k\in \{\pm 1,\dots \pm r^L\}\),
		\begin{equation*}
			P_{e,k}^*(\good_e^k)S_{e,k}^* \geq 1-(aL)^2e^{-2L} - 2a\Cmix (Ll')^2e^{-\cmix Ll'}.
		\end{equation*}Taking \(L\) large enough gives the wanted bounds.
	\end{proof}
	
	\subsection{Concluding the proof}
	
	The last section showed that
	\begin{equation}
		\label{eq:mixing_to_siteEdgePerco}
		\sup_{\xi,\xi'}\dTV\big(\nu_{\Lambda,\Lambda\setminus \Delta_2}^{t,\xi}|_{\Delta_1}, \nu_{\Lambda,\Lambda\setminus \Delta_2}^{t,\xi'}|_{\Delta_1}\big) \leq \rmP([\Delta_1] \text{ is not separated from } [\Delta_2] \text{ in } \tilde{\omega})
	\end{equation}
	where \(\tilde{\omega}\) is a site-edge Bernoulli percolation model on \((\Gamma,E_{\ell})\) with parameters given at the end of Section~\ref{sec:exploration_process}. Taking \(r\geq 2\theta^{-4}\), \(a\geq \frac{4\ln(r)}{\cmix l'}\) fixed, by Claim~\ref{claim:perco_labelled_bnds}, for any \(L\) large enough, this model stochastically dominates a Bernoulli site-edge percolation model \(\omega\) on \((\Gamma,E_{\Gamma})\) with parameters (\(i\in \Gamma\), \(e\in E_{\Gamma}\))
	\begin{gather*}
		q_i = \begin{cases}
			1-r^{-L} & \text{ if } i\in [\Lambda]_{\bulk}\\
			\theta^{(2L_1+1)^2} & \text{ if } i\in [\Lambda]_{\bnd}\\
			1 & \text{ if } i\in [\Lambda]_{\ext}
		\end{cases},
		\\
		q_e = \begin{cases}
			1-(1-\theta^{4L}/2)^{2r^L}\geq 1-e^{-2^L} & \text{ if } e\subset [\Lambda]_{\bulk}\\
			\theta^{(2a+2)L^2} & \text{ if } e=\{i,j\}, i\in [\Lambda], j\notin [\Lambda]_{\bulk}\\
			1 & \text{ else }
		\end{cases}.
	\end{gather*}
	by saying that an edge \(e\in E_{\Gamma}\) is closed only if \(\tilde{\omega}(e,k) = 0\) for all \(k\)'s. Now, by Lemma~\ref{lem:siteEdgePerco_to_sitePerco} (associate all the weight of non-bulk edges to boundary sites and half the weight of each bulk edge to the associated bulk sites), \(\omega\) stochastically dominates a site Bernoulli percolation on \((\Gamma,E_{\Gamma})\) with parameters 
	\begin{equation*}
		p_i = \begin{cases}
			(1-r^{-L})(1-e^{-2^L})^2 & \text{ if } i\in [\Lambda]_{\bulk}\\
			\theta^{(2L_1+1)^2 + 4(2a + 2)L^2} & \text{ if } i\in [\Lambda]_{\bnd}\\
			1 & \text{ if } i\in [\Lambda]_{\ext}
		\end{cases}.
	\end{equation*}Denote \(P_{\Lambda,\bar{p}}\) the law of this percolation model.
	Monotonicity and~\eqref{eq:mixing_to_siteEdgePerco} then give
	\begin{equation}
		\label{eq:mixing_to_sitePerco}
		\sup_{\xi,\xi'}\dTV\big(\nu_{\Lambda,\Lambda\setminus \Delta_2}^{t,\xi}|_{\Delta_1}, \nu_{\Lambda,\Lambda\setminus \Delta_2}^{t,\xi'}|_{\Delta_1}\big) \leq P_{\Lambda,\bar{p}}([\Delta_1] \text{ is not separated from } [\Delta_2] )
	\end{equation}which, up to one use of planar duality, is the statement of Theorem~\ref{thm:main} with \(\ell = K (2l+1) = L_1(2l+1) + (a+1)L(2l+1)\), \(p = 1- (1-r^{-L})(1-e^{-2^L})^2\), \(q = 1-\theta^{(2L_1+1)^2 + 4(2a + 2)L^2}\).
	
	\section{Application I: Finite range Gibbsian specifications}
	\label{sec:applications:FRGibbs}
	
	\subsection*{Mixing in Gibbsian specifications}
	
	The reader in need of more details about Gibbsian formalism is invited to have a look at~\cite[Chapter 2]{Georgii-2011}. Let \(\Omega_0\) be a finite set. For \(A\Subset \Z^2\), let \(\phi_A:\Omega_0^A\to \R\) be a collection of uniformly bounded potential functions. The Gibbs specification associated to the collection \((\phi_A)_{A\Subset \Z^2}\) is the collection of probability kernels \((\mu_{\Lambda}^{\xi})_{\Lambda\Subset \Z^2,\xi\in \Omega_0^{\Z^2}}\) given by, for any event \(A\) with compact support,
	\begin{equation*}
		\mu_{\Lambda}^{\xi}(A) = \sum_{\omega\in \Omega} \mathds{1}_{\omega_{\Lambda^c} = \xi_{\Lambda^c}} \mathds{1}_A(\omega) \frac{1}{Z_{\Lambda}^{\xi}}e^{\sum_{A: A\cap \Lambda\neq \varnothing}\phi_A(\omega) },
	\end{equation*}where \(Z_{\Lambda}^{\xi}\) is the normalization constant making \(\mu_{\Lambda}^{\xi}\) a probability measure. For these kernels to be defined, and in order to have that the one-site marginals have uniformly bounded density with respect to the counting measure on \(\Omega_0\), one requires the following finiteness condition:
	\begin{equation}
		\label{eq:summability_potential_site}
		\sup_{x\in \Z^2}\sum_{x\in A\Subset \Z^2} \norm{\phi_A}_{\infty} <\infty.
	\end{equation}Such a specification is called \emph{finite range} when there is \(R\geq 1\) such that \(\phi_A = 0\) for any \(A\) with diameter larger than \(R\).
	
	\emph{Exponential weak mixing} for the specification is the existence of \(C\geq 0, c>0\) such that for any \(\Delta \subset \Lambda\Subset \Z^2\), and \(\xi,\xi'\in \Omega\)
	\begin{equation}
		\label{eq:weak_mix_Gibbs}
		\dTV\big(\mu_{\Lambda}^{\xi}|_{\Delta}, \mu_{\Lambda}^{\xi'}|_{\Delta}\big) \leq C\sum_{x\in \Delta} \sum_{y\in \Lambda^c} e^{-c|x-y|}.
	\end{equation}
	
	\emph{Exponential strong mixing} for the specification is the existence of \(C\geq 0, c>0\) such that for any \(\Delta \subset \Lambda\Subset \Z^2\), and \(\xi,\xi'\in \Omega\)
	\begin{equation}
		\label{eq:strong_mix_Gibbs}
		\dTV\big(\mu_{\Lambda}^{\xi}|_{\Delta}, \mu_{\Lambda}^{\xi'}|_{\Delta}\big) \leq C\sum_{x\in \Delta} \sum_{y\in \Lambda^c} \mathds{1}_{\xi_y\neq \xi'_y}e^{-c|x-y|}.
	\end{equation}
	
	\subsection*{Weak mixing implies strong mixing}
	
	The main result of this section is a new proof of~\cite[Theorem 1.1]{Martinelli+Olivieri+Schonmann-1994}.
	\begin{theorem}[Martinelli-Olivieri-Schonmann, 1994]
		\label{thm:FRGibbs:weak_implies_strong}
		Suppose that the specification \((\mu_{\Lambda}^{\xi})_{\Lambda,\xi}\) has finite range \(R\) and is exponentially weak mixing. Then, there is \(L_0\geq 0\), and \(C\geq 0,c>0\) such that~\eqref{eq:strong_mix_Gibbs} holds whenever \(\Lambda = x+\{1,\dots,L\}^2\) with \(x\in \Z^2\), \(L\geq L_0\).
	\end{theorem}
	\begin{proof}
		First note that it is sufficient to prove~\eqref{eq:strong_mix_Gibbs} for \(\xi,\xi'\) equal everywhere but at one site \(y\).
		
		The set \((\mu_{\Lambda}^{\xi})_{\Lambda\Subset \Z^2, \xi\in \Omega_0^{\Z^2}}\) is a particular case of the framework of Section~\ref{sec:models}. Take then \(l > R\), the range of the potential. Then, the model satisfies~\ref{hyp:Markov:Decoupling_circuits} with \(\MarkovSet_{u} = \Omega_{0}^{\Lambda_{3l+1}(u)}\), \(\MarkovSet_{u,\Lambda}^{\xi} = \Omega_{0}^{\Lambda_{3l+1}\cap \Lambda}\),~\ref{hyp:Markov:local_Markov_bulk} with \(\pbulk = 1\), and~\ref{hyp:Markov:local_Markov_finite_energy} with \(\MarkovEl_u,\MarkovEl_{u,\Lambda}^{\xi}\) any elements of the sets \(\MarkovSet_u,\MarkovSet_{u,\Lambda}^{\xi}\).
		
		Now, \((\mu_{\Lambda}^{\xi})_{\Lambda,\xi}\) are assumed to be weak mixing. So, by Theorem~\ref{thm:weakMix_to_ratioWeakMix}, they satisfy Hypotheses~\ref{hyp:Mix:inf_vol_meas}, and~\ref{hyp:Mix:exp_rel_density}. Let then \(\ell,q\) be the length and percolation parameter guaranteed by Theorem~\ref{thm:main} for \(p= 10^{-10}\). Then, for \(L\) large enough compared to \(\ell\) (depending on \(q\)), it is an exercise to show that the percolation models \(P_{\Lambda;\ell,\bar{p}}\) have exponential decay of connectivity uniformly over \(\Lambda\) a large enough square (see remark~\ref{rem:general_inhomogeneities} and Lemma~\ref{lem:exp_dec_inhomo_bnd}).
	\end{proof}
	
	\begin{remark}
		\label{rem:general_inhomogeneities}
		The statement is in fact more general: Theorem~\ref{thm:main} compares mixing with connections in a (very) subcritical Bernoulli percolation with inhomogeneities. A sufficient condition for these model to have exponential decay should be that the inhomogeneity is "essentially one-dimensional" below a certain scale (e.g.: it is (well approximated by) a union of straight lines with large distances between intersection points, and lower bounded angles between lines). Obtaining optimal conditions under which such inhomogeneous site/block percolation has exponential decay seems like an interesting open problem. Handling ``macroscopically one dimensional'' inhomogeneities can be done as in~\cite[Lemma 3.2]{Ott+Velenik-2018}. An argument which applies in the large squares case of Theorem~\ref{thm:FRGibbs:weak_implies_strong} is given in Lemma~\ref{lem:exp_dec_inhomo_bnd}.
	\end{remark}
	
	\section{Application II: FK percolation}
	\label{sec:applications:FK}
	
	\subsection*{FK percolation}
	
	Denote \(\bbE_{R}=\big\{\{i,j\}\subset \Z^2:\ \normsup{i-j}\leq R\big\}\) the edges of length less than \(R\). Let \(J_{ij}=J_{ji} =J(i-j)\geq 0\) be such that \(J(x)=0\) when \(\normsup{x}> R\). Let \(q>0\). For \(\Lambda\subset \Z^2\), let \(\bbE_{\Lambda,R} = \{e\in \bbE_R:\ e\cap \Lambda\neq \varnothing \}\). For \(\pi\) a partition of \(\Lambda^c\), define the probability measure \(\Phi_{\Lambda}^{\pi}\) on \(\bbE_{\Lambda,R}\) via
	\begin{equation*}
		\Phi_{\Lambda}^{\pi}(\omega)\propto q^{\kappa_{\pi}(\omega)}\prod_{\{i,j\}\in \omega} (e^{J_{ij}}-1),
	\end{equation*}where \(\kappa_{\pi}(\omega)\) is the number of classes in the partition of \(\Lambda\) obtained by putting in the same class sites that are connected in the graph obtained from \((\Z^2,\omega)\) by identifying vertices in \(\Lambda^c\) that belong to the same class of \(\pi\). Write \(1,0\) for the partitions obtained corresponding, respectively, to the set of connected components of \(\Lambda^c\), and to \(\{\{i\}:\ i\in \Lambda^c\}\).
	See~\cite{Grimmett-2006} for all the properties of FK percolation.
	
	Say that the model is \emph{strongly subcritical} if
	\begin{equation}
		\label{eq:subCrit_FK}
		\lim_{N\to\infty} \sup_{\pi,x}\Phi_{\Lambda_{3N}(x)}^{\pi}\big(\Lambda_N(x)\leftrightarrow \Lambda_{2N}^c(x)\big) = 0.
	\end{equation}
	Denote \(A_{n,m}(x)\) the event that there is at least one cluster in \(\Lambda_m(x)\setminus \Lambda_{n}(x)\) such that this cluster ``surrounds'' \(\Lambda_n(x)\) (it contains a path of open edges such that if the line segments corresponding to its edges are drawn in the plane, the obtained set is included in \(x + ([-m,m]^2\setminus (-n,n)^2)\), and contains a simple path surrounding \((-n,n)^2\)). Denote \(B_n(x)\) the event that there is at most one cluster of euclidean diameter \(\geq \frac{n}{10}\) in \(\Lambda_n(x)\). Say that the model is \emph{strongly supercritical} if 
	\begin{equation}
		\label{eq:superCrit_FK}
		\lim_{N\to\infty} \inf_{\pi,x}\Phi_{\Lambda_{3N}(x)}^{\pi}\big( A_{N,2N}(x)\cap B_{2N}(x) \big) = 1.
	\end{equation}Note that in both cases, the \(\sup,\inf\) over \(x\) is not necessary as the model is translation covariant (\(\Phi_{\Lambda_{3N}(x)}^{\pi}\) is the translation of \(\Phi_{\Lambda_{3N}}^{\pi'}\), with \(\pi'\) the partition obtained by translating \(\pi\)).
	
	\subsection*{FK percolation as a model on sites}

	One can see FK percolation with range \(R\) as a model on sites in multiple ways. Two are presented here.
	
	\vspace*{3pt}
	
	\noindent\textbf{First method}
	
	\vspace*{3pt}
	
	Let \(\psi:\bbE_R\to \Z^2\) be a function that associated to each each one of its endpoints in a translation invariant way (\(\psi(x+e) = x+\psi(e)\)).
	
	Take then the spin space to be
	\begin{equation*}
		\Omega_x = \{0,1\}^{\psi^{-1}(\{x\})}.
	\end{equation*}\(\psi\) induces a natural bijection between this state space and the one of FK percolation, so the family of measures is simply the push-forward by this bijection.
	
	\vspace*{3pt}
	
	\noindent\textbf{Second method}
	
	\vspace*{3pt}
	
	Take single-site spin space to be
	\begin{equation*}
		\Omega_{x} = \{0,1\}^{\{y:\ \normsup{x-y} \leq R\}}.
	\end{equation*}From \(\eta\in \Omega_{\Lambda}\), define \(\omega\in \{0,1\}^{\bbE_{\Lambda,R}}\) via
	\begin{equation*}
		\omega_{e}(\eta) = \begin{cases}
			\eta_i(j)\eta_{j}(i) & \text{ if } e=\{i,j\}\subset \Lambda,\\
			\eta_i(j) & \text{ if } e=\{i,j\},\, i\in \Lambda,\, j\in \Lambda^c.
		\end{cases}
	\end{equation*}

	Define then the measure \(\Psi_{\Lambda}^{\pi}\) on \(\Omega_{\Lambda}\) via
	\begin{multline*}
		\Psi_{\Lambda}^{\pi}(\eta) = q^{\kappa_{\pi}(\eta) }\prod_{i\in \Lambda}\Big(\prod_{j\in \Lambda: \eta_{i}(j) = 1}(e^{J_{ij}}-1)^{1/2}\prod_{j\in \Lambda: \eta_{i}(j) = 0}\big(e^{J_{ij}/2}-(e^{J_{ij}}-1)^{1/2}\big)\Big)\\
		\times \prod_{i\in \Lambda}\prod_{j\in \Lambda^c: \eta_i(j) = 1}(e^{J_{ij}}-1),
	\end{multline*}where \(\kappa_{\pi}(\eta) = \kappa_{\pi}(\omega(\eta))\). It is then a direct computation to check that if \(\eta\sim \Psi_{\Lambda}^{\pi}\), then \(\omega(\eta)\sim \Phi_{\Lambda}^{\pi}\).

	\subsection*{Weak mixing implies strong mixing}
	
	The second application recovers the strong mixing result of~\cite{Alexander-2004} for large squares. It is easy to extend the result for more complicated shapes, but the goal is simply to illustrate how to use Theorem~\ref{thm:main} in the context of percolation models. Note that whilst the results of~\cite{Alexander-2004} are restricted to planar FK-percolation, the present method only assumes finite range.
	
	\begin{theorem}[Strong mixing for squares]
		Suppose that the measures \(\Phi_{\Lambda}^{\pi}\), \(\Lambda\Subset\Z^2\), \(\pi\in \mathrm{Partitions}(\Lambda^c)\), are weak mixing (satisfy~\eqref{eq:weak_mix_Gibbs} with \(\Phi_{\Lambda}^{\pi}\) replacing \(\mu_{\Lambda}^{\xi}\)). Suppose moreover that the model is either strongly supercritical (satisfies~\eqref{eq:superCrit_FK}) or strongly subcritical (satisfies~\eqref{eq:subCrit_FK}). Then, there are \(C\geq 0,c>0,L_0\geq 1\) such that for any \(\Lambda\) translate of \( \{1,\dots, L\}^2\), \(L\geq L_0\), any \(\Delta_1,\Delta_2\subset \bbE_{\Lambda,R}\), and \(\pi \in \{0,1\}\), one has
		\begin{equation*}
			\sup_{\xi,\xi'\in \{0,1\}^{\Delta_1}} \dTV\Big(\Phi_{\Lambda}^{\pi}\big(\omega_{\Delta_2}\in \cdot \given \omega_{\Delta_1} = \xi \big), \Phi_{\Lambda}^{\pi}\big(\omega_{\Delta_2}\in \cdot \given \omega_{\Delta_1} = \xi' \big)\Big)\leq C\sum_{e\in \Delta_1,f\in \Delta_2}e^{-c\rmd(e,f)}.
		\end{equation*}
	\end{theorem}
	\begin{proof}
		One checks the Hypotheses~\ref{hyp:Mix:inf_vol_meas},~\ref{hyp:Mix:exp_rel_density},~\ref{hyp:Markov:Decoupling_circuits},~\ref{hyp:Markov:local_Markov_bulk},~\ref{hyp:Markov:local_Markov_finite_energy} separately in the strongly subcritical and strongly supercritical cases. Then, application of Theorem~\ref{thm:main} to the collection of measure \(\Psi_{\Lambda}^{\pi}\) with \(\pi\in \{0,1\}\) defined in the previous subsection gives the result as in the proof of Theorem~\ref{thm:FRGibbs:weak_implies_strong}. Let \(l> 4R\) (\(R\) the range of the percolation model) be a large enough number.
		
		\vspace*{3pt}
		
		\noindent\textbf{Strongly subcritical case}
		
		\vspace*{3pt}
		
		Following Remark~\ref{rem:MarkovEl_bnd_implies_MarSet_bnd}, we only need to define the sets \(\MarkovSet_u\), and the spin values \(\MarkovEl_{u},\MarkovEl_{u,\Lambda}^t\), \(u\in \bbL_l\). Let \(\MarkovSet_u\) be the event that the configuration \(\eta_{\Lambda_{3l+1}(u)}\) has no connections from \(\Lambda_{l}(u)\) to \(\Lambda_{3l/2}(u)^c\). The elements \(\MarkovEl_{u},\MarkovEl_{u,\Lambda}^t\) are then constant \(0\) configurations. Finite energy gives the wanted properties of \(\MarkovEl_{u},\MarkovEl_{u,\Lambda}^t\), and strong sub-criticality gives the wanted large probability of \(\MarkovSet_u\) once \(l\) is taken large enough.
		
		Hypotheses~\ref{hyp:Mix:inf_vol_meas},~\ref{hyp:Mix:exp_rel_density} are then a direct consequence of Theorem~\ref{thm:weakMix_to_ratioWeakMix} and weak mixing.
		
		\vspace*{3pt}
		
		\noindent\textbf{Strongly supercritical case}
		
		\vspace*{3pt}
		
		The procedure is the same (using strong super-criticality rather than strong sub-criticality) with the following elements and sets: \(\MarkovEl_{u},\MarkovEl_{u,\Lambda}^t\) are the constant \(1\) configurations, and \(\MarkovSet_u\) is the event that there is a path surrounding \(\Lambda_l(u)\) in \(\Lambda_{3l/2}(u)\setminus \Lambda_{l}(u)\), and there is exactly one cluster of diameter greater than \(l/10\) in \(\Lambda_{3l+1}(u)\).
	\end{proof}
	
	\section{Application III: Hard core models}
	\label{sec:applications:HC}
	
	This application seems to be new. As in Section~\ref{sec:applications:FRGibbs}, it is a straightforward consequence of Theorem~\ref{thm:main}, and is almost the same as the application to finite range Gibbsian specifications of Section~\ref{sec:applications:FRGibbs}. So the application is only sketched.
	
	Let \(D\ni 0\) be a fixed connected finite set. A \emph{set of configurations} of the \(D\)-hard core model is
	\begin{equation}
		\Omega^{\HardCore} = \big\{ \eta\in \{0,1\}^{\Z^2}:\ i\neq j \text{ and }\eta_i=\eta_j = 1 \ \implies \ (i+D)\cap (j+D)=\varnothing\big\}.
	\end{equation}
	
	Let \(\mu\in \R\). The specification is given by, for \(\Lambda\Subset \Z^2\), \(\xi\in \Omega^{\HardCore}\),
	\begin{equation*}
		\mu_{\Lambda;\mu}^{\xi}(\eta) = \mathds{1}_{\xi_{\Lambda^c} = \eta_{\Lambda^c}}e^{-\mu \sum_{i\in \Lambda} \eta_i} \prod_{\substack{i \in \Lambda,\, j\in \Z^2\setminus i\\ \eta_j = \eta_i=1}} \mathds{1}_{(i+D)\cap (j+D) =\varnothing}.
	\end{equation*}
	
	As in Section~\ref{sec:applications:FRGibbs}, take \(l\) strictly bigger than the diameter of \(D\). Then,~\ref{hyp:Markov:Decoupling_circuits},~\ref{hyp:Markov:local_Markov_bulk} hold with \(\pbulk = 1\). The only change with respect to Section~\ref{sec:applications:FRGibbs} is that one takes \(\MarkovEl_u = 0\) (constant \(0\) configuration on \(\Lambda_{l}(u)\)) and \(\MarkovEl_{u,\Lambda}^{\xi} = 0\) (constant \(0\) configuration on \(\Lambda_{l}(u)\cap \Lambda\)). Following the same procedure as in Section~\ref{sec:applications:FRGibbs}, one obtains that weak mixing (satisfying~\eqref{eq:weak_mix_Gibbs}) implies strong mixing (satisfying~\eqref{eq:strong_mix_Gibbs}) for volumes that are large enough squares.
	
	\appendix
	
	\section{Bernoulli percolation}
	
	The results/proofs presented here are probably exercises (or even trivialities!) for percolation experts. They are included for completeness.
	
	\subsection{Comparison between edge-site and site percolation}
	
	Let \(G=(V,E)\) be a graph with finite degree. Denote \(i\sim j\) the condition \(\{i,j\}\in E\). For each \(i\in V\), let \(p_i\in [0,1]\), and for each \(\{i,j\}\in E\), let \(p_{ij} \in [0,1]\). Let \(X_{i},i\in V\), \(X_{ij}, \{i,j\}\in E\) be an independent family of Bernoulli random variables with \(P(X_*=1) = p_*\). Consider the random graph
	\begin{equation*}
		G(X) = \big(\{i\in V:\ X_i=1\},\{\{i,j\}\in E:\ X_iX_jX_{ij} =1\}\big).
	\end{equation*}
	
	\begin{lemma}
		\label{lem:siteEdgePerco_to_sitePerco}
		For each \(\{i,j\}\in E\), let \(\lambda_{i,j},\lambda_{j,i}\in [0,1]\) with \(\lambda_{i,j}+\lambda_{j,i} = 1\). Let \(q_i = p_i\prod_{j\sim i} p_{ij}^{\lambda_{i,j}} \) and let \(Y_i\sim \mathrm{Bern}(q_i), i\in V\). Then, the graph
		\begin{equation*}
			H(Y) = \big(\{i\in V:\ X_i = 1 \}, \{\{i,j\}\in E:\ X_iX_j =1 \}\big)
		\end{equation*}is stochastically dominated by \(G(X)\) (for graph inclusion).
	\end{lemma}
	\begin{proof}
		For each oriented edge \((i,j)\) with \(\{i,j\}\in E\), let \(\tilde{p}_{ij} = p_{ij}^{\lambda_{i,j}}\). Let also \(\tilde{p}_i=p_i\) for \(i\in V\). Let \((Z_a)_{a\in V\cup \{(i,j):\ \{i,j\}\in E\}}\) be an independent family with \(X_a\sim \mathrm{Bern}(\tilde{p}_a)\). Then, \(G(X)\) has the same law as \(\tilde{G}(Z) = \big(\{i\in V:\ Z_i = 1 \}, \{\{i,j\}\in E:\ Z_{(i,j)}Z_{(j,i)} = 1\}\big)\). Now, consider the sub-graph \(\tilde{H}(Z)= (W(Z),F(Z))\) of \(\tilde{G}(Z)\) with
		\begin{equation*}
			W(Z) = \Big\{i\in V:\ Z_i\prod_{j\sim i} Z_{(i,j)} = 1\Big\},\quad F(Z)=\big\{\{i,j\}\in E:\ \{i,j\}\subset W(Z) \big\}.
		\end{equation*}By choice of the parameters, \(Z_i\prod_{j\sim i} Z_{(i,j)}\) has the same law as \(Y_i\) for every \(i\in V\). Moreover, the family \(\big(Z_i\prod_{j\sim i} Z_{(i,j)}\big)_{i\in V}\) is an independent family. In particular, the graph \(\tilde{H}(Z)\) has the same law as \(H(Y)\), proving the claim.
	\end{proof}
	
	\subsection{Percolation estimates}
	
	\begin{lemma}
		\label{app:perco:setCluster_to_siteCluster}
		Let \(G=(V,E)\) be a connected graph. Let \(P_p\) be the site Bernoulli percolation measure of parameter \(p\) on \(G\). Let \(X\sim P_p\). Then, for any \(\lambda\geq $  $0\), and any \(A\subset V\),
		\begin{equation*}
			E_p\big(e^{\lambda|C_A(X)|}\big) \leq \prod_{i\in A} E_p\big(e^{\lambda|C_i(X)|}\big)
		\end{equation*}where \(C_A(X)=\cup_{i\in A} C_i(X)\), and \(C_i(X)\) is the cluster of \(i\) in \(X\) (if \(X_i = 0\), \(C_i(X)=\varnothing\)).
	\end{lemma}
	\begin{proof}
		Fix some total order on \(A\), \(A=\{a_1,a_2,\dots\}\). Let \(Y\) be an independent copy of \(X\). The key is that \(|C_A(X)|\) is stochastically dominated by \(|C_{a_1}(X)|+|C_{A\setminus a_1}(Y)|\). Induction then gives the wanted result. Conditionally on \(C_{a_1}(X)\), let \(\overline{C_{a_1}}(X) = \{a_1\}\cup \{y:\ y\leftrightarrow a_1\text{ or } y\sim z, z\leftrightarrow a_1\}\) be the set of sites whose state is fixed by conditioning on \(C_{a_1}(X)\), and \(B = A\cap \overline{C_{a_1}}(X)\). One then has
		\begin{equation*}
			C_{A}(X) = C_{a_1}(X) \sqcup C_{A\setminus B}(X).
		\end{equation*}Moreover, by monotonicity in volume, \(C_{A\setminus B}(X)\) conditionally on \(C_{a_1}(X)\), is stochastically dominated by \(C_{A\setminus B}(Y) \subset C_{A\setminus a_1}(Y)\). This implies that \(|C_{A}(X)|\) is stochastically dominated by \(|C_{a_1}(X)| + |C_{A\setminus a_1}(Y)|\).
	\end{proof}

	\begin{lemma}
		\label{app:perco:moments_clusterSize}
		Let \(d\geq 1\). Let \(G=(V,E)\) be a connected graph of maximal degree \(d\). Let \(P_p\) be the site Bernoulli percolation measure of parameter \(p\) on \(G\). Let \(X\sim P_p\). Then, if \(p\in [0,1], \lambda\in \R\) satisfy \((1-e^{-1})pe^{\lambda} \leq e^{-c_d}\) where \(c_d\geq 0\) is a constant depending only on \(d\) such that the number of connected site-sub-graphs of \(G\) containing a fixed vertex and \(n\) sites is at most \(e^{c_dn}\),
		\begin{equation*}
			\sup_{i\in V} E_p\big(e^{\lambda|C_i|}\big) \leq e.
		\end{equation*}
	\end{lemma}
	\begin{proof}
		As \(G\) has maximal degree \(d\), the number of connected components containing \(i\) with size \(n\) is less than \(e^{c_d n}\) for some \(c_d<\infty\) depending only on \(d\). Then, by a union bound,
		\begin{equation*}
			E_p\big(e^{\lambda|C_i|}\big) \leq \sum_{n\geq 0} p^n e^{\lambda n}e^{c_d n} \leq \sum_{n\geq 0}(1-e^{-1})^{-n} = e,
		\end{equation*}where the last inequality is the assumption on \(p,\lambda\).
	\end{proof}
	
	With these two Lemmas, one can easily control the geometry of large clusters in a box with high density Bernoulli percolation on \(\Z^2\).
	\begin{lemma}
		\label{lem:app:good_box}
		Let \(\alpha\in (0,1)\). Then, for any \(p\) close enough to \(1\), and \(a\geq 1\), there is \(n_0\geq 1\) such that for any \(n\geq n_0\), if \(X\sim P_p\) is a Bernoulli site percolation on \(\Lambda = \llbracket 0,an \rrbracket\times \llbracket 0,n\rrbracket\) equipped with nearest-neighbour edges, one has
		\begin{equation*}
			P_p(A\cap B) \geq 1-e^{-n},
		\end{equation*}where
		\begin{gather*}
			A = \{\text{there is exactly one connected component of size }\geq n/4\},\\
			B= \big\{|\calC\cap \partial^{\rmi}\Lambda| \geq \alpha(2a+2)n\big\}
		\end{gather*}with \(\calC\) the largest connected component of \(\{x:\ X_x=1\}\).
	\end{lemma}
	\begin{proof}
		Not that the closed sites form a percolation \(1-p\) which is as close to \(0\) as wanted. First, the event \(A^c\) implies that at least one vertex belongs to a closed *-cluster of size at least \(n/4\). Applying Lemma~\ref{app:perco:moments_clusterSize} with \(d=8\), \(\lambda = 13\), and using Chebychev's inequality, one obtains that for any \(p\) close enough to \(1\), and \(n\) large enough,
		\begin{equation*}
			P_p(A^c) \leq (an+1)(n+1)e^{-3n} \leq e^{-2n}.
		\end{equation*}
		
		Work now under the assumption that \(A\) is satisfied. Denote \(\calG_x\) the set of vertices closed in \(X\) and \(*\)-connected to \(x\) by closed vertices (in particular, if \(X_x=1\), \(\calG_x = \varnothing\)). Say that a vertex \(x\in \partial^{\rmi}\Lambda\) is \emph{shielded} if it is closed or it is surrounded by a closed \(*\)-path in \(X_{\Lambda}0_{\Lambda^c}\), see Figure~\ref{Fig:Good_Box_Perco}. The number of vertices shielded by the sites of \(\calG_x\) for a given \(x\in \partial^{\rmi}\Lambda\) is at most \(2|\calG_x|\) (as \(\calG_x\) can not cross \(\Lambda\) from one side to the opposite one when working in \(A\)). Now, if one looks at the set of sites of \(\partial^{\rmi}\Lambda\) that are not shielded, they are all connected together, and, provided there is one in two opposite faces, they are all in \(\calC\). In particular, the event \(B^c\cap A\) is included in the event that \(2\sum_{x\in \partial^{\rmi}\Lambda} |\calG_x| \geq (2a+2)n(1-\alpha) \) (recall \(|\partial^{\rmi}\Lambda| = (2a+2)n -4\)). Now applying Lemmas~\ref{app:perco:setCluster_to_siteCluster}, and~\ref{app:perco:moments_clusterSize} with \(d=8,\lambda= \frac{3}{1-\alpha}\), one obtains that for \(p\) close enough to \(1\),
		\begin{equation*}
			E\big(e^{\lambda \sum_{x\in \partial^{\rmi}\Lambda } |\calG_x| }\big)\leq e^{(2a+2)n},
		\end{equation*}in particular, by Chebychev's inequality, and the discussion before the previous display,
		\begin{equation*}
			P(A\cap B^c)\leq e^{-\lambda (a+1)n(1-\alpha)}e^{(2a+2)n} = e^{-(a+1)n}\leq e^{-2n}.
		\end{equation*}This concludes the proof as
		\begin{equation*}
			P(A\cap B) \geq 1- P(A^c) -P(A\cap B^c) \geq 1-2e^{-2n}.
		\end{equation*}
	\end{proof}

	\begin{figure}[h]
		\centering
		\includegraphics[scale=0.8]{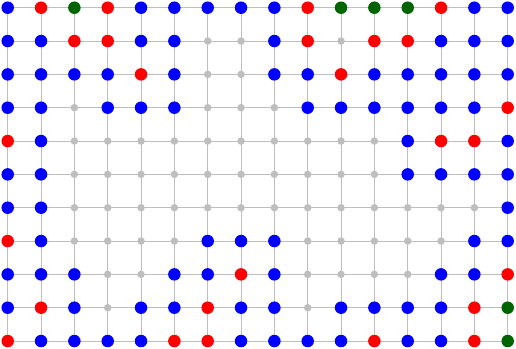}
		\caption{The \(*\)-clusters of closed sites touching the boundary are in red. The blue sites are the union of the non-shielded sites in the boundary with the ``\(*\)-inner boundary'' of the red sites. The green sites are the site that are open but shielded.}
		\label{Fig:Good_Box_Perco}
	\end{figure}
	
	\subsection{Boundary inhomogeneous percolation}
	\label{subsec:exp_dec_inhomo_bnd}
	
	This section contains a proof of exponential decay for (very) subcritical site Bernoulli percolation with inhomogeneities at the boundary of a square box. This is the result used in the Applications.
	
	\begin{lemma}
		\label{lem:exp_dec_inhomo_bnd}
		Let \(p<0.001\), and \(q\in [0,1)\). Then, there are \(c>0,C\geq 0\) such that for any \(n\geq 1\), and any \(x,y\in \Lambda_n \equiv \{-n,\dots, n\}^2\),
		\begin{equation*}
			P_{n; p,q}(x\leftrightarrow_* y)
			\leq
			Ce^{-c|x-y|}
		\end{equation*}where \(P_{n; p,q}\) is the law of a site Bernoulli percolation on \(\Lambda_n\) with parameters \(p\) in \(\Lambda_n \setminus \partial^{\rmi}\Lambda_n\), and \(q\) in \(\partial^{\rmi}\Lambda_n\).
	\end{lemma}
	\begin{proof}
		Let \(X\sim P_{n; p,q}\). Let \(\mathring{X}\) be the restriction of \(X\) to \(\Lambda_n \setminus \partial^{\rmi}\Lambda_n\) (a site percolation Bernoulli with parameter \(p\)). As \(p< 0.001\), one has that there is \(c>0\) such that for any \(u,v\in\Lambda_n \setminus \partial^{\rmi}\Lambda_n\),
		\begin{equation*}
			P\big(u\xleftrightarrow{\mathring{X}}_* v\big)
			\leq
			e^{-c|u-v|}.
		\end{equation*}
		In particular, for any \(x,y\in \Lambda_n\),
		\begin{align*}
			P\big(x\xleftrightarrow{X}_* y\big)
			&\leq
			P\big(x\xleftrightarrow{\mathring{X}}_* y\big) + \sum_{u,v\in \partial^{\rmi}\Lambda_n}P\big(\{x\xleftrightarrow{\mathring{X}}_* \sim u\} \circ\{u \xleftrightarrow{X}_* v\} \circ \{y\xleftrightarrow{\mathring{X}}_* \sim v\} \big)
			\\
			&\leq
			e^{-c|x-y|} + \sum_{u,v\in \partial^{\rmi}\Lambda_n}P\big(x\xleftrightarrow{\mathring{X}}_* \sim u\big) P\big(u \xleftrightarrow{X}_* v\big) P\big(y\xleftrightarrow{\mathring{X}}_* \sim v \big)
			\\
			&\leq
			e^{-c|x-y|} + 9\sum_{u,v\in \partial^{\rmi}\Lambda_n} e^{-c|x-u|} P\big(u \xleftrightarrow{X}_* v\big) e^{-c|y-v|}
		\end{align*}where \(x\leftrightarrow_* \sim y\) is the event that there is a \(*\)-connection from \(x\) to one of the \(*\)-neighbours of \(y\), \(\circ\) denotes disjoint occurrence, we used the BK inequality in the second inequality, and a union bound over the neighbours of \(u,v\) that are in \(\Lambda_n \setminus \partial^{\rmi}\Lambda_n\) in the last. To get the claim, it is sufficient to show that there is \(c>0\) such that for any \(x,y\in \partial^{\rmi}\Lambda_n\),
		\begin{equation*}
			P\big(x \xleftrightarrow{X}_* y\big)\leq e^{-c|x-y|}.
		\end{equation*}
		For \(v\in \partial^{\rmi}\Lambda_n\), let
		\begin{equation*}
			\calC_v = \{x\in \Lambda_n\setminus \partial^{\rmi}\Lambda_n:\ \exists y\in \Lambda_n\setminus \partial^{\rmi}\Lambda_n, \norm{y-v}_{\infty} = 1, x \xleftrightarrow{\mathring{X}}_* y\}.
		\end{equation*}\(\calC_v\) is measurable in terms of \(\mathring{X}\) and is contained in \(\Lambda_n\setminus \partial^{\rmi}\Lambda_n\).
		Define
		\begin{equation*}
			\calC_{\partial} = \cup_{v\in \partial^{\rmi}\Lambda_n}\calC_v.
		\end{equation*}Let now \(x,y\in \partial^{\rmi}\Lambda_n\). Call \(\{u,v\}\subset \partial^{\rmi}\Lambda_n\) a \emph{\((x,y)\)-pre-cut} if \(x,y\) lie in different \(*\)-connected components of \((\partial^{\rmi}\Lambda_n \setminus \{u,v\})\cup \calC_{\partial}\), see Figure~\ref{Fig:ExpDec_Bnd_Box_Perco}.
		\begin{figure}[h]
			\centering
			\includegraphics[scale=0.8]{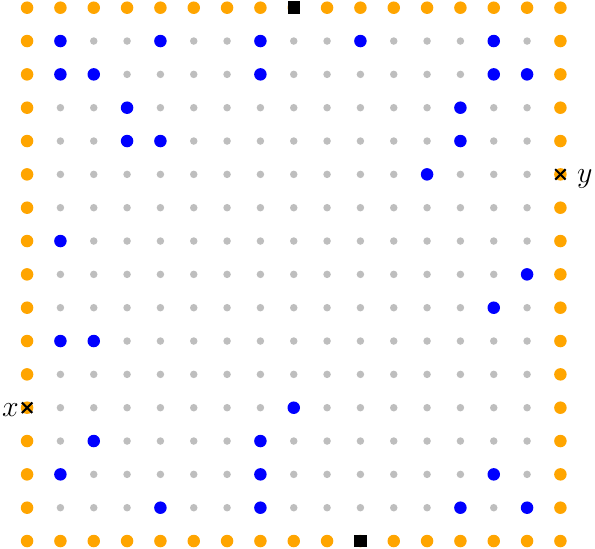}
			\caption{Orange disks: \(\partial^{\rmi}\Lambda_n\), blue disks: \(\calC_{\partial}\), black squares: a \((x,y)\)-pre-cut.}
			\label{Fig:ExpDec_Bnd_Box_Perco}
		\end{figure}
		Observe that conditionally on the existence of at least \(k\) disjoint \((x,y)\)-pre-cuts, the probability of \(x\leftrightarrow_* y\) is at most \((1-(1-q)^2)^k\) as at least one site in each of them has to be open in \(X\) to realize the connection. The wanted result is then reduced to showing that there is at least \(\rho |x-y|\) \((x,y)\)-pre-cuts with probability at least \(1-e^{-c|x-y|}\) for some \(c,\rho >0\).
		Let \(Y_v\), \(v\in \partial^{\rmi}\Lambda_n\), be an i.i.d. collection of \(\Z\)-valued random variables with
		\begin{gather*}
			P(Y_v = k) = P_p\big( X_v = 1, \max\{\norm{u}_{\infty} :\ u\leftrightarrow_* 0\} = k\big) \leq e^{-ck},\ k\geq 0,
			\\
			P(Y_v = -1) = P_p(X_v = 0),
		\end{gather*}where \(P_p\) is a Bernoulli site percolation of parameter \(p\) on \(\Z^2\), and \(c>0\) (this last fact is due to \(p<0.001\)). Define
		\begin{equation*}
			\calD = \bigcup_{v\in \partial^{\rmi}\Lambda_n} (v+\Lambda_{Y_v})
		\end{equation*}where \(\Lambda_{-1} = \varnothing\). This set stochastically dominates \(\calC_{\partial}\). In particular, the wanted bound on the presence of \((x,y)\)-pre-cuts is implied by showing that on each of the two paths going from \(x\) to \(y\) in \(\partial^{\rmi}\Lambda_n\), there are at least \(\rho |x-y|\) points that are not in \(\calD\) with probability at least \(1-e^{-c|x-y|}\) for some \(c,\rho>0\). But this last estimate is a simple one-dimensional percolation estimate for long range stick percolation with exponential tails. The reader in need is invited to look at the Appendix of~\cite{Ott+Velenik-2018} for details on how to implement this kind of argument.
	\end{proof}
	
	\begin{remark}
		The above argument can be improved up to \(p\) smaller than the critical point for \(*\)-percolation. The choice \(0.001\) is arbitrary and sufficient to implement a Peierls argument to obtain exponential bounds for the point-to-boundary percolation probabilities.
	\end{remark}

	\section{Conditional Factorization to Ratio Mixing}
	
	The following result is obtained by repeating the first part of~\cite[Section 5]{Alexander-1998} with only notational changes. This is done in details in the Appendix of~\cite{Dober+Glazman+Ott-2025}.
	\begin{lemma}
		\label{lem:app:mixing_to_ratioMixing}
		Let \(\Omega_i,\, i=1,2\) be finite sets. Let \(\Omega = \Omega_1\times \Omega_2\) and \(\calF_i =\{A\subset \Omega_i\}\), \(\calF = \{A\subset \Omega\}\). Let \(\mu,\nu\) probability measures on \((\Omega,\calF)\). Let
		\begin{equation*}
			\pi_i :\Omega \to \Omega_i,\quad \pi_i((\omega_1,\omega_2)) = \omega_i,\quad
			\mu_i = \mu\circ \pi_i^{-1}.
		\end{equation*}Let \(\epsilon_1,\epsilon_2,\epsilon_3\in [0,1)\). Suppose that all of the following hold:
		\begin{enumerate}
			\item Mixing of \(\mu,\nu\): for every \(\xi,\xi'\in \Omega_1\), \(A\subset \Omega_2\), \(\rho\in \{\mu,\nu\}\)
			\begin{equation*}
				|\rho(\Omega_1\times A\given \{\xi\}\times \Omega_2) - \rho(\Omega_1\times A\given \{\xi'\}\times \Omega_2)|\leq \epsilon_1
			\end{equation*}
			\item Proximity between second marginals: \(\mathrm{d}_{\mathrm{TV}}(\mu_2,\nu_2)\leq \epsilon_2\);
			\item Conditional equality: there exists an event \(D\subset \Omega_2\) such that for \(\rho\in \{\mu_2,\nu_2\}\), \(\rho(D) \geq 1-\epsilon_3\), and for any \(y\in D\),
			\begin{equation*}
				\frac{\mu(x,y)}{\mu_2(y)} = \frac{\nu(x,y)}{\nu_2(y)},\ \forall x\in \Omega_1.
			\end{equation*}
		\end{enumerate}Then, if \(\epsilon = \max(\epsilon_1,\sqrt{\epsilon_2},\epsilon_3) \leq 0.1\), for any \(x\in \Omega_1\) having positive probability under \(\nu_1\),
		\begin{equation*}
			1-9\epsilon \leq \frac{\mu_1(x)}{\nu_1(x)} \leq \frac{1}{1-9\epsilon}.
		\end{equation*}
	\end{lemma}
	
	\bibliographystyle{plain}
	\bibliography{BibTeX}

\begin{thebibliography}{10}

\bibitem{Alexander-1998}
K.~S. Alexander.
\newblock On weak mixing in lattice models.
\newblock {\em Probability theory and related fields}, 110:441--471, 1998.

\bibitem{Alexander-2004}
K.~S. Alexander.
\newblock Mixing properties and exponential decay for lattice systems in finite
  volumes.
\newblock {\em The Annals of Probability}, 32(1A):441--487, 2004.

\bibitem{Ding+Song+Sun-2023}
J.~Ding, J.~Song, and R.~Sun.
\newblock A new correlation inequality for {I}sing models with external fields.
\newblock {\em Probability Theory and Related Fields}, 186(1):477--492, 2023.

\bibitem{Dober+Glazman+Ott-2025}
M.~Dober, A.~Glazman, and S.~Ott.
\newblock Discontinuous transition in {2D} {P}otts: {I}. {O}rder-{D}isorder
  interface convergence.
\newblock {\em arXiv Preprint}.

\bibitem{Dobrushin+Shlosman-1985}
R.~L. Dobrushin and S.~B. Shlosman.
\newblock Completely analytical {G}ibbs fields.
\newblock pages 371--403, 1985.

\bibitem{Dobrushin+Shlosman-1987}
R.~L. Dobrushin and S.~B. Shlosman.
\newblock Completely analytical interactions: constructive description.
\newblock {\em Journal of Statistical Physics}, 46:983--1014, 1987.

\bibitem{Dobrushin+Warstat-1990}
R.~L. Dobrushin and V.~Warstat.
\newblock Completely analytic interactions with infinite values.
\newblock {\em Probability theory and related fields}, 84(3):335--359, 1990.

\bibitem{Georgii-2011}
Hans-Otto Georgii.
\newblock {\em {G}ibbs {M}easures and {P}hase {T}ransitions}.
\newblock DE GRUYTER, 2011.

\bibitem{Grimmett-2006}
G.~Grimmett.
\newblock {\em The random-cluster model}, volume 333 of {\em Grundlehren der
  Mathematischen Wissenschaften [Fundamental Principles of Mathematical
  Sciences]}.
\newblock Springer-Verlag, Berlin, 2006.

\bibitem{Liggett+Schonmann+Stacey-1997}
T.~M. Liggett, R.~H. Schonmann, and A.~M. Stacey.
\newblock Domination by product measures.
\newblock {\em The Annals of Probability}, 25(1):71--95, 1997.

\bibitem{Martinelli-1999}
F.~Martinelli.
\newblock Lectures on {G}lauber dynamics for discrete spin models.
\newblock {\em Lectures on probability theory and statistics (Saint-Flour,
  1997)}, 1717:93--191, 1999.

\bibitem{Martinelli+Olivieri+Schonmann-1994}
F.~Martinelli, E.~Olivieri, and R.~H. Schonmann.
\newblock For {2-D} lattice spin systems weak mixing implies strong mixing.
\newblock {\em Communications in Mathematical Physics}, 165(1):33--47, 1994.

\bibitem{Ott+Velenik-2018}
S.~Ott and Y.~Velenik.
\newblock Potts models with a defect line.
\newblock {\em Comm. Math. Phys.}, 362(1):55--106, Aug 2018.

\bibitem{Schonmann+Shlosman-1995}
R.~H. Schonmann and S.~B. Shlosman.
\newblock Complete analyticity for {2D} {I}sing completed.
\newblock {\em Communications in mathematical physics}, 170:453--482, 1995.

\bibitem{Stroock+Zegarlinski-1992}
D.~W. Stroock and B.~Zegarlinski.
\newblock The logarithmic {S}obolev inequality for discrete spin systems on a
  lattice.
\newblock {\em Communications in Mathematical Physics}, 149:175--193, 1992.

\bibitem{vandenBerg-1999}
J.~van~den Berg.
\newblock On the absence of phase transition in the monomer-dimer model.
\newblock {\em Perplexing Problems in Probability: Festschrift in Honor of
  Harry Kesten}, pages 185--195, 1999.

\bibitem{vanEnter+Fernandez+Schonmann+Shlosman+1997}
A.~CD van Enter, R.~Fern{\'a}ndez, R.~H. Schonmann, and S.~B. Shlosman.
\newblock Complete analyticity of the {2D} {P}otts model above the critical
  temperature.
\newblock {\em Communications in mathematical physics}, 189(2):373--393, 1997.

\bibitem{Yoshida-1997}
N.~Yoshida.
\newblock Relaxed criteria of the {D}obrushin-{S}hlosman mixing condition.
\newblock {\em Journal of statistical physics}, 87:293--309, 1997.

\end{thebibliography}
	
\end{document}